\documentclass[a4paper, 10pt]{article}
\usepackage{amsfonts}
\usepackage{amsmath}
\usepackage{enumerate}
\usepackage{graphicx}
\usepackage[colorlinks=true]{hyperref}
\usepackage{tikz}
\usepackage{enumitem}
\usepackage{cleveref}
\usepackage{multirow}
\usepackage{colortbl}

\newtheorem{proposition}{Proposition}[section]
\newtheorem{theorem}{Theorem}[section]
\newtheorem{lemma}{Lemma}[section]

\newtheorem{definition}{Definition}[section]
\newtheorem{remark}{Remark}[section]

\newenvironment{proof}{\smallskip\noindent\emph{Proof.}\hspace{1pt}}%
{\hspace{-5pt}{\nobreak\quad\nobreak\hfill\nobreak$\square$\vspace{8pt}%
\par}\smallskip\goodbreak}

\newenvironment{proofof}[1]{\smallskip\noindent\emph{Proof of #1.}%
\hspace{1pt}}{\hspace{-5pt}{\nobreak\quad\nobreak\hfill\nobreak%
$\square$\vspace{8pt}\par}\smallskip\goodbreak}

\newcommand{\Lloc}[1]{\mathbf{L^{#1}_{loc}}}
\newcommand{\C}[1]{\mathbf{C^{#1}}}

\newcommand{\Cc}[1]{\mathbf{C_c^{#1}}}

\newcommand{\modulo}[1]{{\left|#1\right|}}
\newcommand{\norma}[1]{{\left\|#1\right\|}}
\newcommand{\reali}{{\mathbb{R}}}

\newcommand{\naturali}{{\mathbb{N}}}
\newcommand{\N}{{\mathbb{N}}}
\newcommand{\tv}{\mathrm{TV}}
\newcommand{\BV}{\mathbf{BV}}
\newcommand{\W}[1]{\mathbf{W^{#1}}}

\newcommand{\eps}{\varepsilon}
\newcommand{\LL}[1]{\mathbf{L^#1}}

\newcommand{\Wloc}[1]{\mathbf{W^{#1}_{loc}}}

\newcommand{\rs}{\mathcal{RS}}

\newcommand{\fw}{$\mathcal{FW}$\,}
\newcommand{\nfw}{$\mathcal{NFW}$\,}
\newcommand{\snfw}{$\mathcal{SNFW}$\,}

\newcommand{\rev}[1]{{#1}}
\newcommand{\revv}[1]{{#1}}

\definecolor{ddgreen}{rgb}{0,0.4,0.4}
\definecolor{dgreen}{rgb}{0,0.65,0}

\begin{document}


\title{Autonomous Vehicles Driving Traffic:\\ The Cauchy Problem}

\author{M. Garavello\thanks{Dipartimento di Matematica e Applicazioni,
    Universit\`a di Milano Bicocca, Via R. Cozzi 55, 20125 Milano, Italy.
    E-mail: \texttt{mauro.garavello@unimib.it}} \and
  F. Marcellini\thanks{INdAM Unit, Dipartimento di Ingegneria dell'Informazione, Universit\`a di Brescia, Via Branze 38, 25123 Brescia, Italy.
    E-mail: \texttt{francesca.marcellini@unibs.it}}}

\maketitle

\begin{abstract}
This paper deals with the Cauchy Problem for a PDE-ODE model, where a system of two conservation laws, namely the Two-Phase macroscopic model proposed
in~\cite{ColomboMarcelliniRascle}, is coupled with an ordinary differential
equation describing the trajectory of an autonomous vehicle (AV), which aims to control the traffic flow. 
Under suitable assumptions, we prove a global in time existence result.

\noindent\textit{2000~Mathematics Subject Classification:} 35L65,
  90B20
 \medskip
 
\end{abstract}

\section{Introduction}

We consider a coupled PDE-ODE system that describes 
the mutual interaction between bulk traffic and 
an autonomous vehicle (AV) responsible to control the traffic flow.
In recent years, the study of autonomous vehicles (AVs) to regulate road traffic developed a lot. The AVs acting as controllers, appear to be the most innovative technology for traffic monitoring and management, also in order to reduce congested traffic and pollution; see~\cite{Davis, GMBE, KFB, TAH, Wang}. 
The main results in this direction are based on ODEs that describe the 
trajectories of both vehicles driven by humans and the AVs, see for instance~\cite{DelleMonache2019275}. More recently, PDE models were coupled with ODEs via a moving flux constraint, see~\cite{DelleMonacheGoatin, DP15} or via AVs see~\cite{MR4002732, GaravelloGoatinLiardPiccoli}.
The novelty of this paper is the use of a Two-Phase model
for modeling the bulk traffic in the road.
Up to now, only the scalar case has been treated in the
literature~\cite{GaravelloGoatinLiardPiccoli}.

The PDE system considered in this paper is the Two-Phase traffic model recently proposed in~\cite{ColomboMarcelliniRascle}. This is a possibly degenerate system of two conservation laws with Lipschitz continuous flow. 
It is coupled with an ordinary differential equation describing the trajectory of the AV.
More precisely, the bulk traffic is governed by the system
\begin{equation}
  \label{eq:2-phase}
  \left\{
    \begin{array}{l}
      \partial_t \rho +
      \partial_x \left( \rho\, v (\rho,w) \right) = 0
      \\
      \partial_t (\rho w) +
      \partial_x \left( \rho w\, v (\rho, w) \right) = 0\,,
    \end{array}
  \right. 
\end{equation}
where $\rho = \rho(t,x) \in [0, R]$ is the car density,
$w = w(t,x) \in [w_{\min},
w_{\max}]$ denotes the maximal
speed of drivers, and $v = v(\rho, w) =
\min \left\{ V_{\max},w \psi(\rho) \right\}$
is the average speed. Here $V_{\max}$ is the maximal
velocity of all the drivers and $\psi=\psi(\rho)$ a decreasing function.
The AV dynamics is described by the control equation
\begin{equation}
  \label{eq:Control-ODE}
  \dot y(t) = \min \left\{u(t), v \left(\rho\left(t, y(t) +\right),
    w\left(t, y(t)+\right)\right) \right\},
\end{equation}
where $y(t)$ denotes the position of the AV at time $t$ and the control
$u = u(t) \in \left[0, V_{\max}\right]$
represents the desired speed. The minimum in~(\ref{eq:Control-ODE})
takes care of the fact that the AV can not 
go faster than the cars immediately in front.
Following~\cite{DelleMonacheGoatin, GaravelloGoatinLiardPiccoli} the complete system is
\begin{equation}
  \label{eq:RPM}
  \left\{
    \begin{array}{l}
      \partial_t \rho +
      \partial_x \left( \rho\, v (\rho,w) \right) = 0
      \\
      \partial_t (\rho w) +
      \partial_x \left( \rho w\, v (\rho, w) \right) = 0
      \vspace{.2cm}\\
      \dot{y}(t) = \min \left\{u(t),
      v \left( \rho(t,y(t)+), w(t,y(t)+) \right)\right\}
      \vspace{.2cm}\\
      \rho\left( t,y(t)\right)
      \left( v(\rho(t,y(t)), w(t,y(t))) - \dot{y}(t)\right)
      \leq F_{\alpha}(w, \dot y(t))\,,    
    \end{array}
  \right. 
\end{equation}
where the last inequality imposes a flux constraint at the
position of the AV through the function
$F_{\alpha}$, which describes the reduced capacity of the road at the
AV position. 
Finally, we couple~(\ref{eq:RPM}) with the initial conditions
\begin{equation}
  \label{eq:IC}
  \left\{
    \begin{array}{l}
      \left(\rho(0, x), w(0, x)\right) = \left(\rho_0(x), w_0(x)\right)
      \\
      y(0) = y_0.
    \end{array}
  \right.
\end{equation}

In this paper, we prove a global in time existence of solutions to the Cauchy Problem for the PDE-ODE model~\eqref{eq:RPM}-\eqref{eq:IC}.
Differently from classical results for hyperbolic conservation
laws, we are able to consider initial conditions with finite total variation, not necessarily small.
\rev{A consequence is that system~\eqref{eq:RPM} can be easily
  generalized to the
  case of several autonomous vehicles provided they do not interact each other}.
\\
Existence
of solutions is obtained through compactness:
we use, for the PDE system,
the Helly's Theorem together with the
wave-front tracking method, while, for the
ODE control equation, the classical Ascoli-Arzelà Theorem.
Here, different from the scalar case where fine estimates on traces
are necessary,
non-characteristic conditions at the AV location
guarantee that the limit is a solution to the Cauchy problem.

The literature on the modeling of vehicular traffic offers a lot of different approaches as macroscopic, microscopic and kinetic models possibly coupled between them. 
In the context of macroscopic models, based on partial differential equations, the basic one is the classical Lighthill--Whitham~\cite{LighthillWhitham} and Richards~\cite{Richards} (LWR) model, given by a single conservation law. Then, the so called \emph{second order} ones are based on two equations, as the
Aw-Rascle-Zhang (ARZ)~\cite{AwRascle, Zhang2002}, the 
GARZ~\cite{FHS13} and the collapsed GARZ (CGARZ) 
model~\cite{2017arXiv170203624F}. A further class, again based on two 
equations, is that of Two-Phase or Phase Transition models 
as~\eqref{eq:2-phase}; they are characterized by two different phases: 
the \textit{Free} and the \textit{Congested one}. A peculiarity of 
two-phase models is the existence of a free regime where only the 
density characterizes the state of the system, while in congested 
regime it is necessary the use of an additional quantity.
Thus, in the free phase the model reduces to a single conservation law,
while in the congested phase it is a hyperbolic system of two conservation laws. For other Two-Phase and Phase Transition models see~\cite{MR4244931, MR4177199, BlandinWorkGoatinPiccoliBayen, Colombo1.5, Goatin2Phases, LebacqueMammarHajSalem}. 
For other kinetic, microscopic or coupled PDE-ODE descriptions see~\cite{MR3253235, MR3019727, ColomboMarcelliniPreprint, MR4357097,  g-p-coupling-2012, GazisHermanRothery, LattanzioPiccoli, MR4172835}.

Up to now, existence of solution to the Cauchy problem
for AV coupled with PDE
has been obtained only in the scalar case, specifically
for the LWR model. A similar, but different, approach consists
in the study of conservation laws with pointwise unilateral constraints on the flow. Their peculiarity is the possible presence of a non-classical shock, violating the classical Kru\v{z}kov~\cite{MR0267257} or Lax~\cite{MR0350216}
entropy admissibility conditions, at the constraint position. Scalar conservation laws with fixed flux constraints have been introduced in~\cite{ColomboGoatin}; here, the problem is to provide a mathematical framework to model local constraints in traffic flow, such as traffic lights or tool gates. Results for the second order models as the Aw-Rascle-Zhang (ARZ) model with fixed constraints are provided in~\cite{ADR, GaravelloGoatin, GaravelloVilla}. Problems with moving constraints have been considered in~\cite{DelleMonacheGoatin} for the scalar case and~\cite{VGC} for the ARZ model.

The article is organized as follows. Section~\ref{sec:basic-assumptions} gives a description of the model from an analytic point of view and describes the solution of the \rev{constrained} Riemann problem.
Section~\ref{sec:CMC} contains the proof of the existence result for the Cauchy problem; the proof is divided into four different subsections.
Finally an appendix with technical lemmas concludes the
paper.

\section{Basic Properties and the Riemann Problem}
\label{sec:basic-assumptions}
In this section we recall basic properties of the Two-Phase model
and the Riemann problem both in the classical case
and in presence of flux constraints.
For a detailed description of~(\ref{eq:2-phase}) we refer
to~\cite{ColomboMarcelliniRascle}.
The free and congested phases are described by the sets
\begin{align*}
  F
  & = 
  \left\{
    (\rho,w) \in [0,R] \times [w_{\min}, w_{\max}]
    \colon v(\rho, \rho w) = V_{\max}
  \right\},
  \\
  C
  & = 
  \left\{
    (\rho,w) \in [0,R] \times [w_{\min}, w_{\max}]
    \colon v(\rho, \rho w) = w \, \psi(\rho)
  \right\}\,,
\end{align*}
represented in Figure~\ref{fig:phases}.
Here we assume the following conditions. 
\begin{enumerate}[label=\textbf{(H-\arabic*)}, ref=\textup{\textbf{(H-\arabic*)}}, align=left]
\item \label{Hyp:H1} $R, w_{\min}, w_{\max}, V_{\max}$ are positive
  constants, with $V_{\max} < w_{\min} < w_{\max}$; $R$ is the maximal
  possible density, typically $R=1$.

\item \label{Hyp:H2} $\psi \in \C{2} \left( [0,R];[0,1]\right)$ is such that
  {$\psi(0) = 1$}, $\psi(R) = 0$, and, for every $\rho \in [0,R]$,
  $c_\psi \le -\psi'(\rho) \le C_\psi$,
  $\frac{d^2\ }{d\rho^2} \left( \rho\, \psi(\rho) \right) \le 0$ for
  suitable constants $0 < c_\psi < C_\psi$.

\item\label{Hyp:H3} Waves of the first family in the congested phase C have
  negative speed.
  More precisely, we assume that there exists a positive constant
  $\bar \lambda$ such that $\lambda_1 \left(\rho, w\right) \le - \bar \lambda$,
  where $\lambda_1 \left(\rho, w\right) = w\left(\rho\psi'(\rho)
    + \psi(\rho)\right)$
  is the first eigenvalue of the Jacobian matrix of the
  flux.


\item\label{Hyp:H4}
  There exist $L_F > 0$ and
  $F_{\alpha,1} \in \C1\left([w_{\min}, w_{\max}]; \reali^+\right)$
  satisfying:
  \begin{enumerate}
  \item $F_\alpha\left(w, \sigma\right) = F_{\alpha,1}(w) \left(V_{\max} -
      \sigma\right)$ for every $w \in [w_{\min}, w_{\max}]$
    and $\sigma \in [0, V_{\max}]$,
    where $F_\alpha$ is the function appearing in~\eqref{eq:RPM};

  \item the inequality
    \begin{equation*}
      \modulo{F_{\alpha, 1}(w_1) - F_{\alpha, 1}(w_2)}
      \le L_F \modulo{w_1 - w_2}
    \end{equation*}
    holds for every $w_1, w_2 \in [w_{\min}, w_{\max}]$;
    
  \item $\psi\left(F_{\alpha, 1}(w_{\max})\right) > \frac{V_{\max}}{w_{\min}}$;
    
  \item $F_{\alpha, 1}'(w) \ge 0$ for every $w \in [w_{\min}, w_{\max}]$.
  \end{enumerate}
\end{enumerate}


\begin{figure}
  \centering
  \begin{tikzpicture}[line cap=round,line join=round,x=1.cm,y=1.cm]

    \draw[<->] (.5, 4.8) -- (.5, .2) -- (5.5, .2);
    \node[anchor = west, inner sep = 0] at (0.6, 4.8) {$\rho v$};
    \node[anchor = west, inner sep = 0] at (5.6, .4) {$\rho$};
    \draw[color=red, fill] (2., 0.2 + 2.5 * 2 - 2.5 * 0.5)
    to [out = -20, in = 100] (5., 0.2)
    to [out = 120, in = -20] (1.5, 0.2 + 2.5 * 1.5 - 2.5 * 0.5)
    -- (2., 0.2 + 2.5 * 2 - 2.5 * 0.5);
    \draw[domain=.5:2., smooth, variable=\x, dgreen, samples = 2,
    line width = 1.3pt] plot (\x, {0.2 + 2.5*\x - 2.5*0.5} );
    \draw[line width = 1.1pt] (2., 0.2 + 2.5 * 2 - 2.5 * 0.5)
    to [out = -20, in = 100] (5., 0.2);
    \draw[line width = 1.1pt] (1.5, 0.2 + 2.5 * 1.5 - 2.5 * 0.5)
    to [out = -20, in = 120] (5., 0.2);
    
    \node[anchor = south, inner sep = 0] at (1., 2.) {$F$};
    \node[inner sep = 0] at (2.6, 3.) {$C$};
    \node[anchor = north, inner sep = 0] at (5., .1) {$R$};

    \draw[<->] (.5 + 7, 4.8) -- (.5 + 7, .2) -- (5.5 + 7, .2);
    \node[anchor = west, inner sep = 0] at (0.6 + 7, 4.8) {$\rho v$};
    \node[anchor = west, inner sep = 0] at (5.6 + 7, .4) {$\rho$};
    \draw[domain=.5:2., smooth, variable=\x, samples = 2]
    plot (\x + 7, {0.2 + 2.5*\x - 2.5*0.5} );
    \draw (2.+7, 0.2 + 2.5 * 2 - 2.5 * 0.5) to [out = -20, in = 100]
    (5.+7, 0.2);
    \draw (1.5+7, 0.2 + 2.5 * 1.5 - 2.5 * 0.5) to [out = -20, in = 120]
    (5.+7, 0.2);
    \node[anchor = north, inner sep = 0] at (5.+7, .1) {$R$};

    \draw (0.5 + 7, 1.7) -- (4.5+7, 3.5);
    \draw (1.7+7, 0.2 + 2.5 * 1.7 - 2.5 * 0.5)
    to [out = -20, in = 115] (5.+7, 0.2);
    \draw[dashed] (1.22+7, 2.03) -- (1.22+7, 0.2);
    \draw[dashed] (2.75+7, 2.73) -- (2.75+7, 0.2);

    \draw[dashed] (8.5, 2.7) -- (8.5, .2);

    \node[anchor = north, inner sep = 0] at (8.6, .05) {$\rho_c$};

    \node[anchor = north, inner sep = 0] at (1.1+7, .1) {$\check \rho$};

    \node[anchor = north, inner sep = 0] at (2.75+7, .1) {$\hat \rho$};
    \node[anchor = east, inner sep = 0] at (7.4, 1.7) {\small
      $F_\alpha(w, \sigma)$};

  \end{tikzpicture}

  \caption{Left, the free phase $F$ (in green) and the congested phase $C$
    (in red) in the coordinates $(\rho, \rho v)$.
    Right, the construction of $\check \rho$, $\hat \rho$, and $\rho_c$.}
  \label{fig:phases}
\end{figure}
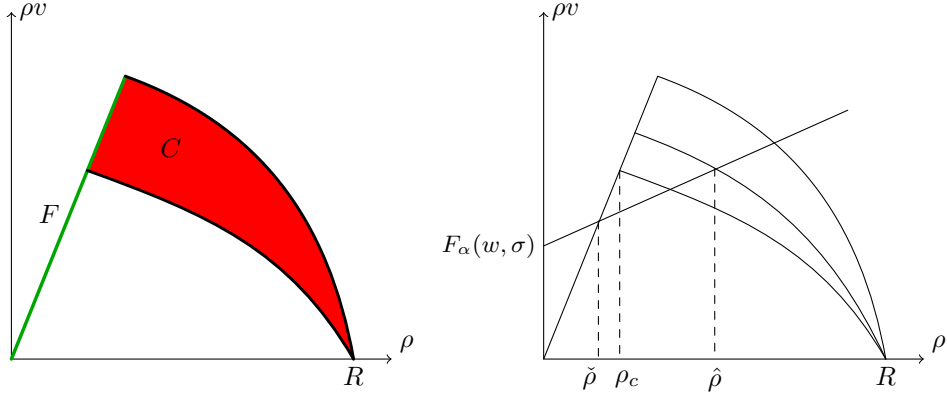

Note that assumptions~\ref{Hyp:H1}-\ref{Hyp:H2} imply that
there exists a unique value $\rho_c \in (0, R)$ such that
$\psi(\rho_c) = \frac{V_{\max}}{w_{\min}}$; see~\Cref{fig:phases} right.

\begin{remark}
  \label{rmk:velocity}
  In the congested phase, the variable $w$ is constant along the curves
  of the first family.
  Similarly, the variable $v$ is constant along the curves of the
  second family. In particular, this implies that $v$ and $w$ are Riemann
  invariants. Thus,
  system~\eqref{eq:2-phase} can be represented alternatively through the
  couple $(v, w)$.
\end{remark}

For $w \in [w_{\min}, w_{\max}]$ and $\sigma \in [0, V_{\max}]$,
define the function
\begin{equation*}
  \label{eq:phi}
  \begin{array}{rcc}
    \varphi_{w, \sigma}:[0, R]
    & \longrightarrow
    & \reali
    \\
    \rho
    & \longmapsto
    &
      F_\alpha(w, \sigma) + \rho \sigma
  \end{array}
\end{equation*}
and the unique densities $\check \rho = \check \rho\left(w, \sigma\right)$
and $\hat \rho = \hat \rho\left(w, \sigma\right)$ respectively as the solutions
to
\begin{equation*}
  \varphi_{w, \sigma}\left(\check\rho\right) = \check\rho V_{\max}
\end{equation*}
and to
\begin{equation*}
  \varphi_{w, \sigma}\left(\hat\rho\right) = w \hat\rho \psi(\hat\rho).
\end{equation*}
Moreover
\begin{equation}
  \label{eq:check-rho}
  \check \rho \left(w, \sigma\right) = \frac{F_\alpha\left(w, \sigma\right)}
  {V_{\max} - \sigma}
  = F_{\alpha, 1}(w) 
\end{equation}
and so, if $w_1 < w_2$ and $\sigma \in [0, V_{\max}]$, then by~\ref{Hyp:H4}
\begin{equation}
  \label{eq:lip-fhi-check}
  \begin{split}
    0
    & \le \check \rho \left(w_2, \sigma\right)
    - \check \rho \left(w_1, \sigma\right)
    \le {L_F} 
    \left(w_2 - w_1\right).
  \end{split}
\end{equation}
By~\ref{Hyp:H4}, the function $\check \rho(w, \sigma)$ is
increasing with respect to $w$ and constant in $\sigma$.
Note also that, by assumption (c) in~\ref{Hyp:H4},
$\left(\check \rho, w\right) \in F \setminus C$
for every $w \in [w_{\min}, w_{\max}]$.
In principle the point $\left(\hat \rho, w\right)$ not necessarily belongs
to $C$; however, in the following, when the flux constraint is effective, then
$\left(\hat \rho, w\right)$ is in $C$; see Figure~\ref{fig:phases}, right.

\subsection{Classical and \rev{Constrained} Riemann Problems}

First, we denote by
\begin{equation*}
  \begin{array}{rccc}
    \mathcal{RS} \colon
    &
      \left(F \cup C\right)^2
    &
      \longrightarrow
    &
      \Lloc1 (\reali; F \cup C)
    \\
    &
      \left(\left(\rho_l, w_l\right), \left(\rho_r, w_r\right)\right)
    &
      \longmapsto
    &
      \left(\mathcal{RS}_\rho, \mathcal{RS}_w\right)
  \end{array}
\end{equation*}
the Riemann solver for the classical Riemann problem
\begin{equation}
  \label{eq:2-phase-RP}
  \left\{
    \begin{array}{l}
      \partial_t \rho +
      \partial_x \left( \rho\, v (\rho,w) \right) = 0,
      \\
      \partial_t (\rho w) +
      \partial_x \left( \rho w\, v (\rho, w) \right) = 0\,,
      \\
      \left(\rho, w\right)(0, x) = \left\{
      \begin{array}{l@{\qquad}l}
        \left(\rho_l, w_l\right),
        & x < 0,
        \\
        \left(\rho_r, w_r\right),
        & x > 0,
      \end{array}
      \right.
    \end{array}
  \right. 
\end{equation}
in the sense that the functions
\begin{equation*}
  \rho(t, x) = \mathcal{RS}_\rho \left(\frac{x}{t}\right),
  \qquad
  w(t, x) = \mathcal{RS}_w \left(\frac{x}{t}\right),
\end{equation*}
provide the solution to the Riemann problem~(\ref{eq:2-phase-RP}).

At the location of the AV, we have to consider a \rev{constrained} Riemann
problem, which depends also on the speed of the AV; see also~\cite{MR4060810}. 
\begin{definition}
  Fix a constant $\bar u \!\in\! \left[0, V_{\max}\right]$ and two states
  $\left(\rho_l, w_l\right)\!, \left(\rho_r, w_r\right) \!\in F \cup C$.
  The Riemann problem for~(\ref{eq:RPM}), related to
  $\bar u$ and to $\left(\rho_l, w_l\right)$ and $\left(\rho_r, w_r\right)$,
  is the following Cauchy problem
  \begin{equation}
    \label{eq:constraint-RP}
    \left\{
      \begin{array}{l}
        \partial_t \rho +
        \partial_x \left( \rho\, v (\rho,w) \right) = 0,
        \\
        \partial_t (\rho w) +
        \partial_x \left( \rho w\, v (\rho, w) \right) = 0,
        \vspace{.2cm}\\
        \dot{y}(t) = \min \left\{\bar u,
        v \left( \rho(t,y(t)+), w(t,y(t)+) \right)\right\},
        \vspace{.2cm}\\
        \rho\left( t,y(t)\right)
        \left( v(\rho(t,y(t)), w(t,y(t))) - \dot{y}(t)\right)
        \leq F_{\alpha}(w, \dot y(t)),
        \vspace{.2cm}\\
        \left(\rho, w\right) (0, x) =
        \left\{
        \begin{array}{l@{\qquad}l}
          \left(\rho_l, w_l\right),
          & x < 0,
          \\
          \left(\rho_r, w_r\right),
          & x > 0,
        \end{array}
            \right.
        \\
        y(0) = 0.
      \end{array}
    \right. 
  \end{equation}
\end{definition}
The solution to the \rev{constrained} Riemann problem~\eqref{eq:constraint-RP}
is given through the \rev{constrained} Riemann solver, introduced
in the following definition.
\begin{definition}
  \label{def:Riemann-solver}
  The \rev{constrained} Riemann solver
  \begin{equation*}
    \begin{array}{rccc}
      \mathcal{RS}^{c} \colon
      &
        \left(F \cup C\right)^2 \times [0, V_{\max}]
      &
        \longrightarrow
      &
        \Lloc1 (\reali; F \cup C) \times [0, V_{\max}]
      \\
      &
        \left(\left(\rho_l, w_l\right), \left(\rho_r, w_r\right), \bar u\right)
      &
        \longmapsto
      &
        \left(\left(\mathcal{RS}^c_{\rho}, \mathcal{RS}^c_{w}\right),
        \boldsymbol{u}\right)
    \end{array}
  \end{equation*}
  is a function representing a solution to the \rev{constrained} Riemann
  problem~(\ref{eq:constraint-RP}), in the sense that the functions
  \begin{equation*}
    \rho\left(t, x\right) = \mathcal{RS}^c_{\rho}\left(\frac{x}{t}\right),
    \qquad
    w\left(t, x\right) = \mathcal{RS}^c_{w}\left(\frac{x}{t}\right),
    \qquad
    y(t) = \boldsymbol{u}\, t
  \end{equation*}
  are a solution to~(\ref{eq:constraint-RP}).
\end{definition}

Denoting by $f_1\left(\rho, w\right) = \rho v \left(\rho, w\right)$,
the construction of $\mathcal{RS}^c$ is as follows:
\begin{enumerate}
\item if
  $f_{1}(\mathcal{RS}((\rho_{l},w_{l}), (\rho_{r},w_{r}))(\bar u))
  \!\leq\! \varphi_{w_l, \bar u} (\mathcal{RS}_{\rho}((\rho_{l},w_{l}),
  (\rho_{r},w_{r}))(\bar u))$, then
  \begin{equation*}
    \begin{array}{@{}l@{}}
      \displaystyle
      \mathcal{RS}^{c}_{\rho}
      \left( (\rho_{l},w_{l}), (\rho_{r},w_{r}), \bar u\right)
      \left(\frac{x}{t}\right)
      =
      \mathcal{RS}_{\rho}
      \left( (\rho_{l},w_{l}), (\rho_{r},w_{r})\right) \left(\frac{x}{t}\right),
      \\
      \displaystyle
      \mathcal{RS}^{c}_{w}
      \left( (\rho_{l},w_{l}), (\rho_{r},w_{r}), \bar u\right) \left(\frac{x}{t}
      \right)
      =
      \mathcal{RS}_{w}
      \left( (\rho_{l},w_{l}), (\rho_{r},w_{r})\right) \left(\frac{x}{t}\right),
      \\
      \boldsymbol{u} = \min \left\{\bar u, v \left(\mathcal{RS}
      \left(\left(\rho_l, w_l\right), \left(\rho_r, w_r\right)\right)
      \left(\bar u +\right)\right)\right\};
    \end{array}
  \end{equation*}
  
\item if
  $f_{1}(\mathcal{RS}((\rho_{l},w_{l}), (\rho_{r},w_{r}))(\bar u))
  \!>\! \varphi_{w_l, \bar u} (\mathcal{RS}_{\rho}((\rho_{l},w_{l}),
  (\rho_{r},w_{r}))(\bar u))$, then
  \begin{equation*}
    \begin{array}{@{}l@{}}
      \displaystyle
      \mathcal{RS}^{c}_{\rho}
      \left( (\rho_{l},w_{l}), (\rho_{r},w_{r}), \bar u\right)
      \!\!\left(\frac{x}{t}\right)
      \!=\!
      \left\{
      \!\!\!
      \begin{array}{ll}
        \mathcal{RS}_{\rho}\!
        \left( (\rho_{l},w_{l}), (\hat \rho,w_{l})\right)
        \left(\frac{x}{t}\right),
        &
          \!\frac{x}{t} < \bar u,
        \vspace{.2cm}\\
        \mathcal{RS}_{\rho}\!
        \left( (\check \rho, w_{l}), (\rho_r, w_{r})\right)
        \left(\frac{x}{t}\right),
        &
          \!\frac{x}{t} > \bar u,
      \end{array}\!\!
          \right.
      \\
      \displaystyle
      \mathcal{RS}^{c}_{w}\!
      \left( (\rho_{l},w_{l}), (\rho_{r},w_{r}), \bar u\right)
      \!\!\left(\frac{x}{t}
      \right)
      \!=\!
      \left\{
      \!\!\!
      \begin{array}{ll}
        \mathcal{RS}_{w}\!
        \left( (\rho_{l},w_{l}), (\hat \rho,w_{l})\right)
        \!\left(\frac{x}{t}\right),
        &
          \!\!\frac{x}{t} < \bar u,
        \vspace{.2cm}\\
        \mathcal{RS}_{w}\!
        \left( (\check \rho, w_{l}), (\rho_r, w_{r})\right)
        \!\left(\frac{x}{t}\right),
        &
          \!\!\frac{x}{t} > \bar u,
      \end{array}
          \right.
      \\
      \boldsymbol{u} = \bar u.
    \end{array}
  \end{equation*}
\end{enumerate}

\begin{remark}
  The Riemann solver $\rs^c$ produces a solution to the \rev{constrained} Riemann
  problem~\eqref{eq:constraint-RP} such that the density flux at the AV
  location is below the constraint imposed by such vehicle.
  More precisely, it returns the ``classical'' solution in the case its flux
  is below the threshold, otherwise it produces \rev{a non classical wave}
  connecting two states, which satisfy the flux constraint, traveling
  with speed of the AV.
\end{remark}

It is interesting to note that the \rev{constrained} Riemann solver, introduced
in~\Cref{def:Riemann-solver}, satisfies a consistency property.
\begin{proposition}
  Fix $\bar u \in \left[0, V_{\max}\right]$ and two states
  $\left(\rho_l, w_l\right)\!, \left(\rho_r, w_r\right) \!\in F \cup C$
  and define $\boldsymbol{u}$ as the velocity component of
  $\rs^c\left(\left(\rho_l, w_l\right),
    \left(\rho_r, w_r\right), \bar u\right)$.
  Then
  \begin{equation}
    \label{eq:RCc-consistency}
    \rs^c\left(\left(\rho_l, w_l\right),
      \left(\rho_r, w_r\right), \bar u\right)
    = \rs^c\left(\left(\rho_l, w_l\right),
      \left(\rho_r, w_r\right), {\boldsymbol u}\right).
  \end{equation}
\end{proposition}
\begin{proof}
  If $f_{1}(\mathcal{RS}((\rho_{l},w_{l}), (\rho_{r},w_{r}))(\bar u))
  > \varphi_{w_l, \bar u} (\mathcal{RS}_{\rho}((\rho_{l},w_{l}),
  (\rho_{r},w_{r}))(\bar u))$, then
  ${\boldsymbol u} = \bar u$ and so~\eqref{eq:RCc-consistency} clearly holds.

  If $f_{1}(\mathcal{RS}((\rho_{l},w_{l}), (\rho_{r},w_{r}))(\bar u))
  \le \varphi_{w_l, \bar u} (\mathcal{RS}_{\rho}((\rho_{l},w_{l}),
  (\rho_{r},w_{r}))(\bar u))$, then
  \begin{equation*}
    \boldsymbol{u} = \min \left\{\bar u, v \left(\mathcal{RS}
        \left(\left(\rho_l, w_l\right), \left(\rho_r, w_r\right)\right)
        \left(\bar u +\right)\right)\right\}
    \le \bar u.
  \end{equation*}
  Clearly in the case $\boldsymbol{u} = \bar u$, the conclusion easily
  follows. Assume therefore that $\boldsymbol{u} < \bar u$.
  Define by $\left(\rho_m, w_m\right)$ as the middle state (if it exists)
  in the solution
  of the classical Riemann problem with left state $\left(\rho_l, w_l\right)$
  and right state $\left(\rho_r, w_r\right)$, otherwise put
  $\left(\rho_m, w_m\right) = \left(\rho_l, w_l\right)$.
  Then define
  \begin{equation*}
    \left(\bar \rho, \bar w\right) \in
    \left\{\left(\rho_l, w_l\right), \left(\rho_m, w_m\right),
      \left(\rho_r, w_r\right)\right\}
  \end{equation*}
  such that ${\boldsymbol u} = v\left(\bar \rho, \bar w\right)$.
  Note that $\left(\bar \rho, \bar w\right)$ exists by the construction
  in~\Cref{def:Riemann-solver}. Since ${\boldsymbol u} < \bar u \le V_{\max}$,
  then $\left(\bar \rho, \bar w\right) \in C$ and so
  \begin{equation*}
    f_1 \left(\bar \rho, \bar w\right) = {\boldsymbol u} \bar \rho
    < F_\alpha\left(w_l, {\boldsymbol u}\right) + {\boldsymbol u} \bar \rho
    = \varphi_{w_l, {\boldsymbol u}} (\mathcal{RS}_{\rho}((\rho_{l},w_{l}),
  (\rho_{r},w_{r}))(\bar u)).
  \end{equation*}
  This permits to conclude.
\end{proof}

\subsection{Wave Notation}
Below, we list the waves and the notations that we will use in the present
paper.

\begin{itemize}
\item \textsl{First Family Wave} \rev{$1$}: a wave connecting a left state 
  $\left(\rho_l, w_l\right) \in C$ with a right state
  $\left(\rho_r, w_r\right) \in C$ such that
  $w_{l} =w_{r}$.
  
\item \textsl{Second Family Wave} \rev{$2$}: a wave connecting a left state 
  $\left(\rho_l, w_l\right) \in C$ with a right state
  $\left(\rho_r, w_r\right) \in C$ such that
  $v\left(\rho_l, w_l\right) = v\left(\rho_r, w_r\right)$.  
  
\item \textsl{Linear Wave} \rev{$\mathcal{LW}$}:
  a wave connecting two states in the free
  phase.

\item \textsl{Phase Transition Wave} \rev{$\mathcal{PT}$}:
  a wave connecting a left state
  $\left(\rho_l, w_l\right) \in F$ with a right state
  $\left(\rho_r, w_r\right) \in C$ satisfying 
  $w_{l} =w_{r}$.
  
\item \textsl{Fictitious Wave} \fw: a wave denoting the AV
  trajectory without discontinuity in $(\rho, w)$. The notation stands for
  a fictitious wave.

\item \textsl{Non Fictitious Wave} \nfw: a wave denoting the AV trajectory
  with discontinuity in $(\rho, w)$
  \revv{and connecting the left state $\left(\rho_l, w_l\right)$
    with the right state $\left(\rho_r, w_r\right)$ such that
    $w_l = w_r$, $\rho_l = \hat \rho\left(w_l, \sigma\right)$, and
    $\rho_r = \check \rho\left(w_l, \sigma\right)$ for some
    $\sigma \in [0, V_{\max}[$}. The notation stands for a non
  fictitious wave.
  
\item \revv{\textsl{Special Non Fictitious Wave} \snfw: a wave denoting
    the AV trajectory
    with discontinuity in $(\rho, w)$, which is not a \nfw.}
\end{itemize}

\begin{remark}
  We note that the states $(\rho_{l},w_{l})$ and $(\hat{\rho}, w_l)$
  are connected by a possible combination of waves of the first family
  and phase transition waves with speed less than $\bar u$.
  The states $(\check{\rho}, w_l)$ and $(\rho_{r},w_{r})$ are connected
  either by a linear wave or by a possible combination of a phase transition
  and a second family wave with speed greater than $\bar u$.
\end{remark}

\begin{remark}
  \revv{\nfw and \snfw are both waves denoting the AV trajectory.
    In the former case, it is a wave where the flux constraints holds with
    equality, i.e. $\rho(v - \dot y) = F_\alpha$ with $\dot y < V_{\max}$.
    In the latter case, the AV trajectory coincides with a classical
    wave of the Two-Phase model.}
\end{remark}
\section{The Cauchy Problem}
\label{sec:CMC}

In this section we consider the Cauchy problem for the control
problem~(\ref{eq:RPM}) with moving constraint, that is
\begin{equation}
  \label{eq:CP}
  \left\{
    \begin{array}{l}
      \partial_t \rho +
      \partial_x \left( \rho\, v (\rho,w) \right) = 0
      \\
      \partial_t (\rho w) +
      \partial_x \left( \rho w\, v (\rho, w) \right) = 0
      \\
      \dot{y}(t) = \min \left\{u(t), v
      \left( \rho(t,y(t)+), w(t,y(t)+) \right)\right\}
      \\
      \rho\left( t,y(t)\right)
      \left( v(\rho(t,y(t)), w(t,y(t))) - \dot{y}(t)\right)
      \leq F_{\alpha}(w, \dot y(t))  
  \\ 
  \left( \rho,w\right)(0,x)=\left( \rho_{0}(x),w_{0}(x)\right)
  \\
  y(0)=y_{0}\,,
    \end{array}
  \right. 
\end{equation}
with control function $u \in \BV\left(\reali^+; [0, V_{\max}]\right)$,
initial data $\left(\rho_0, w_0\right): \reali \to F \cup C$,
and $y_{0} \in \reali$. Before stating the main result,
we introduce the definition of solution to the \rev{constrained} Cauchy
problem~(\ref{eq:CP}).
\begin{definition}
  \label{def:weak-solution-1}
  The couple
  \begin{displaymath}
    \left( (\rho^{*}, w^{*}),y^{*} \right)  \in \C0 \left( [0, +\infty[;
      \LL1((\reali; F \cup C) \right) \times
    \Wloc{1,\infty} \left( [0, +\infty[; \reali \right) 
  \end{displaymath}
  is a solution to~(\ref{eq:CP}) with control \rev{$u = u(t)$}, if
  \begin{enumerate}
  \item the function $(\rho^{*}, w^{*})$ is a weak solution (see \Cref{rmk:weak-sol-2phase}) to the PDE in~(\ref{eq:CP}), for $(t,x) \in \Omega_-$ and for $(t,x) \in \Omega_+$, where
  \begin{align*}
      \Omega_-
      & = \left\{(t,x) \in (0, +\infty) \times \reali: x < y^*(t) \right\}
      \\
      \Omega_+
      & = \left\{(t,x) \in (0, +\infty) \times \reali: x > y^*(t) \right\};
  \end{align*}

  \item for a.e. $t>0$, the function $x \mapsto (\rho^{*}(t,x), w^{*}(t,x))$ has bounded total variation;

   \item $\left(\rho^{*}(0, x), w^{*}(0, x)\right) = \left(\rho_0(x), w_0(x)\right)$, for a.e. $x \in \reali$;

   \item  the function $y$ is a Caratheodory solution to the
     ODE in~(\ref{eq:CP}), i.e. for a.e. $t \in \reali^{+}$ 
       \begin{equation*}
         y(t) = y_{0} +
         \int_{0}^{t}\min \left\{u(s),
           v \left(s,y(s)+\right)\right\}\,ds\,;
       \end{equation*}

   \item  the constraint is satisfied, in the sense that for a.e. $t \in \reali^{+}$
\begin{equation*}
  \lim_{x \rightarrow y(t)\pm}\,\left(  \rho\left(t, x\right)
    \left( v(\rho(t,x), w(t,x)) - \dot{y}(t)\right)
    - F_{\alpha}\left(w\left(t, x\right), \dot y(t)\right)\right)
  \leq 0\,.
  \end{equation*}
  \end{enumerate}
\end{definition}

\begin{remark}
  Note that the point 4. of \Cref{def:weak-solution-1} is formulated directly
  using the variable $v$, accordingly with \Cref{rmk:velocity}.
\end{remark}

\begin{remark}
  \label{rmk:weak-sol-2phase}
  By~\cite[Remark~5.3]{ColomboMarcelliniRascle},
  a couple $\left(\rho^*, w^*\right)$ is a weak solution to the PDE
  in~\eqref{eq:CP} if and only if the couple $\left(\rho^*, \eta^*\right)$,
  with $\eta^* = \rho^* w^*$,
  is a weak solution to
  \begin{equation*}
    \left\{
      \begin{array}{l}
        \partial_t \rho +
        \partial_x \left( \rho\, v (\rho, \frac{\eta}{\rho}) \right) = 0
        \\
        \partial_t \eta +
        \partial_x \left( \eta\, v (\rho, \frac{\eta}{\rho}) \right) = 0.
      \end{array}
    \right.
  \end{equation*}
\end{remark}

We can now state the main result of the paper:
\begin{theorem}
  \label{thm:cauchy-problem}
  Let assumptions~\textup{\ref{Hyp:H1}}, \textup{\ref{Hyp:H2}},
  \textup{\ref{Hyp:H3}}, and~\ref{Hyp:H4} hold.
  Fix the control function $u \in \BV \left(\reali^+; [0, V_{\max}]
  \right)$ and the initial conditions
  $(\rho_0, w_0) \in \BV (\reali; F \cup C)$ and $y_{0} \in \reali$.
  Then there exists $\left((\rho^{*}, w^{*}),y^{*}\right)$, a solution to~(\ref{eq:CP}) in the sense of Definition~\ref{def:weak-solution-1}.
\end{theorem}
The proof is contained in the following subsections. In particular, we construct a sequence of approximate solutions by using the wave-front tracking method, and we prove its convergence.
Throughout the section, we implicitly assume that
hypotheses~\ref{Hyp:H1}--\ref{Hyp:H4} hold.


\subsection{Wave-Front Tracking Approximate Solution}
\label{sse:wft-12}
In this subsection we construct piecewise constant approximations via the wave-front tracking method, which is a set of techniques to obtain approximate solutions to hyperbolic conservation laws in one space dimension. It was first introduced by Dafermos~\cite{MR2574377}, see also~\cite{BressanLectureNotes, MR3443431} for the general theory. 

At first, we give the following definition of an $\eps$-approximate wave-front tracking solution to~\eqref{eq:CP}.

\begin{definition}
  \label{def:epswf}
  Given $\eps > 0$, the map $\left(z_\eps,
    y_{\eps}, u_\eps\right)$ is an $\eps$-approximate
  wave-front tracking solution to~\eqref{eq:CP} if the following conditions
  hold.
  \begin{enumerate}
  \item  $z_\eps = \left(\rho_\eps, w_\eps\right) \in
    \C0 \left( [0, +\infty[; \LL1(\reali; F \cup C) \right)$;
    
  \item  $y_\eps \in \W{1,\infty} \left( [0, +\infty[; \reali\right)$;
    
  \item  $u_\eps \in \BV\left( [0, +\infty[; [0, V_{\max}]\right)$ is
    piecewise constant; 
  
  \item $\left(\rho_{\eps}, w_{\eps}\right)$
    is piecewise constant, with discontinuities along finitely
    many straight lines in $(0, +\infty) \times \reali$.
    Moreover the jumps can be of the first family,
    of the second family, linear waves or phase transition waves;
  
  \item it holds that
    \begin{equation*}
      \left\{
        \begin{array}{l}
          \norma{\left(\rho_{\eps}(0,\cdot),
               w_{\eps}(0,\cdot) \right)
            - \left( \rho_0(\cdot),w_0(\cdot)\right)}_{\LL1 (\reali)}
          <\eps 
          \\
          \tv \, \left(\rho_{\eps}(0,\cdot), w_{\eps}(0,\cdot)
          \right)
          \le \tv \, \left( \rho_0(\cdot),w_0(\cdot)\right)
          \\
          \norma{u_\eps - u}_{\LL1\left(\reali^+\right)} < \eps
          \\
          \tv \left(u_\eps\right) \le \tv \left(u\right)\,;
        \end{array}
      \right.
    \end{equation*}
  \item for a.e. $t \in \reali^{+}$,
    \begin{equation}
      \label{eq:velocity-AV-wft}
      y_{\eps}(t) = y_{0}
      + \int_{0}^{t} \min \left\{u_\eps(s),
        v\left( \rho_\eps(s,y_{\eps}(s)+), w_\eps
          (s,y_{\eps}(s)+)\right)
      \right\}\,ds;
    \end{equation}
\item  the constraint is satisfied,
  in the sense that for a.e. $t \in \reali^{+}$
\begin{equation*}
  \begin{split}
    &
    \lim_{x \rightarrow y_{\eps}(t)\pm}\,\left(
      \rho_\eps\left( t,y_{\eps}(t)+\right)
      \left( v(\rho_\eps(t,y_{\eps}(t)+),
        w_\eps(t,y_{\eps}(t)+)) - \dot{y_{\eps}}(t)\right)\right) (t,x)
    \\
    & \qquad
    \leq F_{\alpha}\left(w_\eps\left(t, y_\eps(t)+\right),
      \dot y_\eps(t)\right)\,.
  \end{split}
\end{equation*}
\end{enumerate}
\end{definition}

We describe here a procedure for constructing a sequence of wave-front
approximate solutions. For every $\nu \in \naturali \setminus \left\{0\right\}$,
we consider the triple $(\rho_{0,\nu}, w_{0,\nu}, u_\nu)$
of piecewise constant functions with a finite number of discontinuities
such that the following conditions hold.
\begin{enumerate}
\item $(\rho_{0,\nu}, w_{0,\nu}):\reali \to F \cup C$
  and $u_\nu: \reali^+ \to [0, V_{\max}]$;

\item the following limits hold
  \begin{align*}
    \lim_{\nu\to+\infty} (\rho_{0,\nu}, w_{0,\nu})  =
    (\rho_{0}, w_{0}) & \qquad \textrm{ in } \LL1(\reali; F \cup C)
    \\
    \lim_{\nu\to+\infty} u_\nu  =
    u & \qquad \textrm{ in } \LL1(\reali^+; [0, V_{\max}]);
  \end{align*}

\item the following inequalities hold
  \begin{align*}
    \tv (\rho_{0,\nu}, w_{0,\nu}) & \le \tv (\rho_{0}, w_{0})
    \\
    \tv \left(u_{\nu}\right) & \le \tv \left(u\right).
  \end{align*}
\end{enumerate}

Next, for every $\nu \in \N \setminus \{0\}$, we proceed with the
following method. At time $t=0$, we solve all the classical Riemann problems for
$x \in \reali$, $x \ne y_0$ and the \rev{constrained} Riemann problem, located at
$y_0$. We approximate every rarefaction wave of the first family with a
rarefaction fan, formed by rarefaction shocks of strength less than
$\frac{1}{\nu}$ traveling with the Rankine-Hugoniot speed.
In this way we construct a piecewise approximate solution
$\left(\rho_\nu, w_\nu, y_\nu\right)$ until the first time at which
two waves interact together, or a wave interacts with the AV, or
$u_\nu$ has a discontinuity.
In the first case, we solve the classical Riemann problem and we prolong
the approximate solution beyond this time.
In the second case or in the third one,
we solve the \rev{constrained} Riemann problem and we prolong
the approximate solution beyond this time.
We repeat this procedure at every interaction times.

\begin{remark}
  \label{rmk:1}
  Slightly changing the velocity of waves,
  \rev{as described in~\cite[Remark~7.1]{BressanLectureNotes}},
  or the discontinuity times of $u_\nu$,
  we may assume that, at every positive time $t$, at most one
  of the following possibilities
  happens:
  \begin{enumerate}
  \item two classical waves (first family wave, second family wave, linear wave,
    and phase transition wave) interact together at a point $x \in \reali$,
    with $x \ne y_\nu(t)$;

  \item a classical wave hits the AV trajectory $y_\nu$;
    
  \item $u_\nu\left(t-\right) \ne u_\nu(t+)$.
  \end{enumerate}
\end{remark}

\begin{remark}
  \label{rmk:no-superposition}
  \revv{The $\mathcal{NFW}$ wave connects two states through a
    non-classical shock with positive speed.}

  \revv{Instead, the \snfw connects two states through a wave with
    positive speed, which can be either a second family wave or a
    phase transition wave $\mathcal{PT}$, or a linear wave
    $\mathcal{LW}$}.

  \revv{It is always possible to construct a wave-front tracking
    approximate solution such that, for $t$ in a right neighborhood of
    $0$, the $\mathcal{SNFW}$ wave is not present.}

  \revv{Moreover, the interaction estimates considered in \Cref{2}
    imply that, at positive times, the $\mathcal{SNFW}$ wave can be
    only superimposed to a linear wave.}
\end{remark}

Given an $\eps$-approximate wave-front tracking solution
$\left(z_\eps, y_\eps, u_\eps\right)$, define, for a.e. $t > 0$, the following functionals
\begin{align}
  \label{eq:funct_F_w}
  \mathcal F_w (t)
  & =
    \sum_{x \in \reali} \modulo{w\left(z_\eps(t, x+)\right)
    - w\left(z_\eps(t, x-)\right)}
  \\
  \nonumber
  \mathcal F_{\tilde v} (t)
  &
    =
    \sum_{x \in \reali} \modulo{\tilde v\left(z_\eps(t, x+)\right)
    - \tilde v\left(z_\eps(t, x-)\right)}
  \\
  \label{eq:funct_F_v}
  & \quad
    - 2 \modulo{\tilde v\left(z_{\eps}\left(t, y_{\eps}(t)+\right)\right)
    - \tilde v\left(z_{\eps}\left(t, y_{\eps}(t)-\right)\right)}
    \\
  \label{eq:w_funct_PT}
  \mathcal F(t) & =  \mathcal F_w (t) + \mathcal F_{\tilde v} (t)\,,
  \\
  \label{eq:numbers-wave-func}
  \mathcal N(t)
                   & = \# \left\{x \in \reali: z_\eps\left(t, x-\right) \ne
                     z_\eps\left(t, x+\right)\right\},
\end{align}
where, we denote by $\tilde v$ the function
\begin{equation*}
  \tilde v(\rho, w) = w \psi(\rho).
\end{equation*}
Moreover the functionals $\mathcal N^-_1(t)$ and $\mathcal N^+_1(t)$
denote respectively the number of discontinuities given by a wave of the
first family at the left, resp. at the right, of $y_\eps(t)$.
Finally, $\mathcal N^-_2(t)$, $\mathcal N^-_2(t)$,
$\mathcal N^-_{\mathcal{PT}}(t)$, $\mathcal N^-_{\mathcal{PT}}(t)$,
$\mathcal N^-_{\mathcal{LW}}(t)$, $\mathcal N^-_{\mathcal{LW}}(t)$
are defined similarly.
Note that the previous functionals may vary only at times $\bar t$,
described in Remark~\ref{rmk:1}.

The functional $\mathcal F(t)$ is composed by $2$ terms. The first
term measures the strength of waves of second family. The second term
measures the strength of waves of first family and of phase transition
waves. Moreover both of the first two terms measure the strength of
linear waves.
\rev{Note that the functional $\mathcal F_{\tilde v}$ contains a term
  which depends on the position of the AV. This is just a technical tool
  for having better interaction estimates; see~\cite{ADR} for a
  similar approach. In general this functional may assume negative values;
  however it is bounded from below,
  since its second term is uniform bounded.}

\subsection{Interaction estimates}
\label{2}

We consider interactions estimates between waves.
We will consider different types of interactions separately.
It is not restrictive to assume that, at any interaction time $t=\bar t$,
exactly one possibility enumerated in Remark~\ref{rmk:1} happens:
either two waves interact, or a wave hits the
AV trajectory, or the control changes. 

We describe wave interactions by the nature of the involved waves, see~\cite{GaravelloPiccoli, g-p-coupling-2012}. For example, if a wave of the second family hits a wave of the first family producing a phase-transition wave, we write \textbf{$2$-$1$/ $\mathcal{PT}$}. Here the symbol ``$/$'' divides the waves before and after the interaction. 

\subsubsection{Collisions between classical waves}

For the classical collision between two waves, we have the following result.

\begin{lemma}
  \label{le:estim-funct-w-interaction-PT}
  Assume that the wave $((\rho^l, w^l), (\rho^m, w^m))$
  interacts with the wave
  $((\rho^m, w^m), (\rho^r, w^r))$ at the point $(\bar t,
  \bar x)$ with $\bar t > 0$ and $\bar x \in \reali$.  Then
  $\mathcal F (\bar t+) \le \mathcal F(\bar t-)$.
  The possible interactions are: \textbf{$2$-$1$/$1$-$2$},
  \textbf{$\mathcal{LW}$-$\mathcal{PT}$/$\mathcal{PT}$-$2$},
  \textbf{$1$-$1$/$1$}, \textbf{$\mathcal{PT}$-$1$/$\mathcal{PT}$}.
  Hence $\mathcal N(\bar t+) \le \mathcal N(\bar t-)$.
\end{lemma}

See~\cite[Lemma 3.4.]{Marcellini2} for the proof.

\subsubsection{Collisions with a fictitious wave}
In this part, we assume that a wave $((\rho^l, w^l), (\rho^r, w^r))$
interacts at a time $\bar t$ with the AV and that there is no discontinuity
at the position of the AV before time $\bar t$.
For simplicity, we introduce also the following notations:
\begin{equation}
  \label{eq:simplicity-1}
  \begin{array}{r@{\,=\,}l@{\quad}r@{\,=\,}l@{\quad}r@{\,=\,}l@{\quad}
    r@{\,=\,}l}
    u^-
    & u\left(\bar t\right),
    & u^+
    & u^-,
    & \check \rho
    & \check \rho ( u^-,w^l),
    & \hat \rho
    & \hat \rho ( u^-,w^l),
    \vspace{.2cm}\\
    v^l
    & \tilde v(\rho^l, w^l),
    & v^r
    & \tilde v(\rho^r, w^r),
    & \check{v}
    & \tilde v(\check{\rho}, {w^l}),
    & \hat{v}
    & \tilde v(\hat{\rho}, {w^l}),
  \end{array}
\end{equation}
and we denote with $\xi^-$ and $\xi^+$ respectively the speed of the
AV before and after the interaction (see~\eqref{eq:velocity-AV-wft}),
and define $\Delta \xi = \xi^+ - \xi^-$.

We have the following results.
\begin{lemma}
  \label{Y1}
  Assume that the second family wave $((\rho^l, w^l), (\rho^r, w^r))$
  interacts from the left
  with the \fw wave at the point $(\bar t,\bar x)$ with $\bar t > 0$
  and $\bar x \in \reali$. We have the following cases:
  \begin{enumerate}
  \item No new wave is produced.
    Then $\Delta \mathcal F_w (\bar t)
    = \Delta \mathcal F_{\tilde v} (\bar t) = \Delta \mathcal N (\bar t) = 0$,
    so that $\Delta \mathcal F (\bar t) = 0$.
    Finally $\modulo{\Delta \xi} \le \modulo{v^l - v^r}$.

  \item The interaction is of type
    \textbf{$2$-$\mathcal{FW}$/$1$-$\mathcal{NFW}$-$\mathcal{PT}$-$2$}.
    Then we deduce that $\Delta \mathcal F_w (\bar t) = \Delta \mathcal F_{\tilde v} (\bar t)
    = \Delta \mathcal F (\bar t) = 0$,
    and $\Delta \mathcal N\left(\bar t\right) = 3$.
    Finally $\modulo{\Delta \xi} = 0$.
  \end{enumerate}
\end{lemma}

\begin{proof}
  We use the notation in~(\ref{eq:simplicity-1}).
  Since before $\bar t$ there is no discontinuity at the position of the AV,
  then $\rho^r v^r \leq \varphi_{w^r, u^-}\left(\rho^r\right)$.
  At time $\bar t$, we need to consider the Riemann solver
  \begin{equation*}
    \mathcal{RS}^c \left((\rho^l, w^l),
      \left(\rho^r, w^r\right), u^-\right)
  \end{equation*}
  and the AV enters the region with state $(\rho^l, w^l)$.
  We have two cases.
  \begin{enumerate}    
  \item $\rho^l v^l \leq \varphi_{w^l, u^-}(\rho^l)$.
    In this case, the second family wave crosses the AV trajectory
    and no new wave is created. Thus,
    for the functionals~(\ref{eq:funct_F_w})-~(\ref{eq:numbers-wave-func}),
    we have: 
    \begin{equation*}
      \Delta \mathcal F_w (\bar t) = \Delta \mathcal F_{\tilde v} (\bar t)
      = \Delta \mathcal F (\bar t)
      = \Delta \mathcal N (\bar t)
      = 0\,.
    \end{equation*}
    Moreover $\xi^- = \min\{v^r, u^-\}$, $\xi^+ = \min\{v^l, u^-\}$,
    and so $\modulo{\Delta \xi} \le \modulo{v^l - v^r}$.

  \item $\rho^l v^l > \varphi_{w^l, u^-}(\rho^l)$.
    In this case, the second family wave crosses the AV trajectory producing
    a first family shock wave $((\rho^l, w^l), (\hat{\rho}, w^l))$,
    a $\mathcal{NFW}$ wave $((\hat{\rho}, w^l), (\check{\rho}, w^l))$,
    a phase transition wave with positive speed
    $((\check{\rho}, w^l), (\rho^l, w^l))$, and
    a second family wave $((\rho^l, w^l), (\rho^r, w^r))$, so that
    $\Delta \mathcal N(\bar t) = 3$.
    For the functional~(\ref{eq:funct_F_w}) we have
    $\Delta \mathcal F_w (\bar t) = 0$.

    For the functional~(\ref{eq:funct_F_v}), since $v^l = v^r$ because
    the interacting wave is a second family wave and
    since $v^l > \hat{v}$, $\check{v}> \hat{v}$ and $\check{v}>v^l$,
    we have that
    \begin{align*}
      \Delta \mathcal F_{\tilde v} (\bar t)
      &
        = \modulo{v^l - \hat{v}} + \modulo{\check{v} - \hat{v}}
        + \modulo{\check{v} - v^l}
        + \modulo{v^l - v^r} - 2 \modulo{\check{v} - \hat{v}}
        - \modulo{v^l - v^r}\\
      & = 0\,.
    \end{align*}
    Here we have $\xi^- = \min\{v^r, u^-\} = u^-$, $\xi^+ = u^-$, and so
    $\Delta \xi = 0$.
  \end{enumerate}
\end{proof}

\begin{lemma}
  \label{Y2}
  Assume that the phase transition wave
  $((\rho^l, w^l), (\rho^r, w^r))$ interacts from the left with
  the \fw wave at the point $(\bar t,\bar x)$ with $\bar t > 0$
  and $\bar x \in \reali$.
  Then the interacting phase transition wave has positive speed and
  no new wave is produced. Moreover $\Delta \mathcal F_w (\bar t)
  = \Delta \mathcal F_{\tilde v} (\bar t)
  =\Delta \mathcal F(\bar t) = 0$ and $\Delta \mathcal N(\bar t) = 0$.
  Finally $\modulo{\Delta \xi} \le \modulo{v^l - v^r}$.
\end{lemma}

\begin{proof}
  We use the notation in~(\ref{eq:simplicity-1}).
  Since the interacting wave is a phase transition, then $w^l = w^r$.
  Before $\bar t$ there is no discontinuity at the position of the AV,
  then $\rho^r v^r \leq \varphi_{w^r, u^-}(\rho^r)$.
  
  The speed of the phase transition
  is bigger than that of the fictitious wave. In particular it is positive and
  $\rho^l v^l \leq \varphi_{w^l, u^-}(\rho^l)$.
  Thus, at time $\bar t$, the solution of
  \begin{equation*}
    \mathcal{RS}^c \left((\rho^l, w^l),
      \left(\rho^r, w^r\right), u^-\right)
  \end{equation*}
  is classical, in the sense that the phase transition wave crosses
  the AV trajectory and no new wave is created. Thus,
  for the functionals~(\ref{eq:funct_F_w})-(\ref{eq:numbers-wave-func}),
  we have: 
  \begin{equation*}
    \Delta \mathcal F_w (\bar t) = \Delta \mathcal F_{\tilde v} (\bar t)
    = \Delta \mathcal F (\bar t)
    = \Delta \mathcal N (\bar t)
    = 0\,.
  \end{equation*}
  Here $\xi^- = \min\{v^r, u^-\}$, $\xi^+ = \min\{V_{\max}, u^-\} = u^-$,
  and so either $\modulo{\Delta \xi} = 0$ or $\modulo{\Delta \xi} = u^- - v^r
  \le \modulo{v^l - v^r}$.
\end{proof}

\begin{lemma}
  \label{Y2Bis}
  Assume that the \fw wave interacts from the left with the phase transition wave $((\rho^l, w^l), (\rho^r, w^r))$ at the point $(\bar t,\bar x)$ with $\bar t > 0$ and $\bar x \in \reali$.
  Then no new wave is produced. Moreover $\Delta \mathcal F_w (\bar t)
  = \Delta \mathcal F_{\tilde v} (\bar t)
  =\Delta \mathcal F(\bar t) = 0$ and $\Delta \mathcal N(\bar t) = 0$.
  Finally $\modulo{\Delta \xi} \le \modulo{v^l - v^r}$.
\end{lemma} 

\begin{proof} 
  We use the notation in~(\ref{eq:simplicity-1}).
  Since the interacting wave is a phase transition, then $w^l = w^r$.
  Before $\bar t$ there is no discontinuity at the position of the AV,
  then $\rho^l v^l \leq \varphi_{w^l, u^-}(\rho^l)$.
  
  The speed of the phase transition
  is smaller than that of the fictitious wave. In particular
  $\rho^r v^r \leq \varphi_{w^l, u^-}(\rho^r)$.
  Thus, at time $\bar t$, the solution of
  \begin{equation*}
    \mathcal{RS}^c \left((\rho^l, w^l),
      \left(\rho^r, w^r\right), u^-\right)
  \end{equation*}
  is classical, in the sense that the phase transition wave crosses
  the AV trajectory and no new wave is created. Thus,
  for the functionals~(\ref{eq:funct_F_w})-(\ref{eq:numbers-wave-func}),
  we have: 
  \begin{equation*}
    \Delta \mathcal F_w (\bar t) = \Delta \mathcal F_{\tilde v} (\bar t)
    = \Delta \mathcal F (\bar t)
    = \Delta \mathcal N (\bar t)
    = 0\,.
  \end{equation*}
  Here $\xi^- = \min\{V_{\max}, u^-\} = u^-$, $\xi^+ = \min\{v^r, u^-\}$,
  and so either $\modulo{\Delta \xi} = 0$ or $\modulo{\Delta \xi} = u^- - v^r
  \le \modulo{v^l - v^r}$.
\end{proof} 

\begin{lemma}
  \label{Y3}
  Assume that the linear wave $((\rho^l, w^l), (\rho^r, w^r))$ interacts with the \fw wave at the point $(\bar t,\bar x)$ with $\bar t > 0$ and $\bar x \in \reali$.
  We have the two different cases:
  \begin{enumerate}
  \item No new wave is produced. Then $\Delta \mathcal F_w (\bar t) = \Delta \mathcal F_{\tilde v} (\bar t)=
    \Delta \mathcal F (\bar t)=\Delta \mathcal N (\bar t) = 0$.
    Finally $\Delta \xi = 0$.

  \item The interaction is
    \textbf{$\mathcal{LW}$-\fw/$\mathcal{PT}$-\nfw-$\mathcal{LW}$}.
    Then we get $\Delta \mathcal F_w (\bar t)= 0$,
    $\Delta \mathcal F(\bar t) = \Delta \mathcal F_{\tilde v} (\bar t)
    \le 0$, and
    $\Delta \mathcal N(\bar t) = 2$.
    Finally $\Delta \xi = 0$.
  \end{enumerate} 
\end{lemma}

\begin{proof}
  We use the notation in~(\ref{eq:simplicity-1}).
  Before $\bar t$ there is no discontinuity at the position of the AV,
  then $\rho^r v^r \leq \varphi_{w^r, u^-}\left(\rho^r\right)$.
  At time $\bar t$, we need to consider
  \begin{equation*}
    \mathcal{RS}^c \left((\rho^l, w^l),
      \left(\rho^r, w^r\right), u^-\right).
  \end{equation*}
  We have two different cases.
  \begin{enumerate}
  \item $\rho^l v^l \leq \varphi_{w^l, u^-}(\rho^l)$.
    In this case, the linear wave crosses the bus trajectory and
    no new wave is created.
    Thus, for the
    functionals~(\ref{eq:funct_F_w})-(\ref{eq:numbers-wave-func}), we have: 
    \begin{equation*}
      \Delta \mathcal F_w (\bar t) = \Delta \mathcal F_{\tilde v} (\bar t)
      = \Delta \mathcal F (\bar t)
      = \Delta \mathcal N (\bar t)
      = 0\,.
    \end{equation*}
    Here $\xi^- = \xi^+ = \min\{V_{\max}, u^-\} = u^-$ and so
    $\Delta \xi = 0$.

  \item $\rho^l v^l > \varphi_{w^l, u^-}(\rho^l)$.
    In this case the linear wave crosses the AV trajectory
    producing a phase transition wave
    $((\rho^l, w^l), (\hat{\rho}, w^l))$,
    a \nfw wave $((\hat{\rho}, {w^l}), (\check{\rho}, w^l))$,
    and a linear wave $((\check{\rho}, {w^l}), (\rho^r, w^r))$.
    For the
    functionals~(\ref{eq:funct_F_w})-(\ref{eq:numbers-wave-func}),
    since $\check v > v^l > \hat v$, we have:
    \begin{align*}
      \Delta \mathcal F_w (\bar t)
      & = \modulo{w^l - w^r} - \modulo{w^l - w^r} = 0\,,
      \\
      \Delta \mathcal F_{\tilde v} (\bar t)
      & = \modulo{v^l - \hat{v}} + \modulo{\check{v}-\hat{v}}
        +\modulo{\check{v}-v^r} - 2\modulo{\check{v}-\hat{v}} - \modulo{v^l-v^r}
      \\
      & = \left(v^l - \hat v\right) - \left(\check v - \hat v\right)
        + \modulo{\check v - v^r} - \modulo{v^l - v^r}
      \\
      & \le 0
      \\
      \Delta \mathcal N (\bar t)
      & = 2\,.
    \end{align*}
    Here $\xi^- = \min\{V_{\max}, u^-\} = u^-$, $\xi^+ = u^-$ and so
    $\Delta \xi = 0$.
  \end{enumerate}
\end{proof}

\begin{lemma}
  \label{Y4}
  Assume that the \fw wave interacts with the first family wave
  $((\rho^l, w^l), (\rho^r, w^r))$ at the point $(\bar t,\bar x)$ with
  $\bar t > 0$ and $\bar x \in \reali$. We have the following cases:
  \begin{enumerate}
  \item No new wave is produced.
    Then $\Delta \mathcal F_w (\bar t)= \Delta \mathcal F_{\tilde v} (\bar t)
    = \Delta \mathcal F(\bar t) = \Delta \mathcal N(\bar t) =0$.
    Finally $\modulo{\Delta \xi} \le \modulo{v^l - v^r}$.

  \item The interaction is \textbf{\fw-$1$/$1$-\nfw-$\mathcal{PT}$}.
    Then $\Delta \mathcal F_w (\bar t)= 0$,
    $\Delta \mathcal F(\bar t) = \rev{\Delta} \mathcal F_{\tilde v} (\bar t)
    < 0$, and
    $\Delta \mathcal N(\bar t) = 2$.
    Finally $\modulo{\Delta \xi} \le v^r - v^l$.
  \end{enumerate} 
\end{lemma}

\begin{proof}
  We use the notation in~(\ref{eq:simplicity-1}).
  Before $\bar t$ there is no discontinuity at the position of the AV,
  then $\rho^l v^l \leq \varphi_{w^l, u^-}(\rho^l)$.
  Moreover $w^l = w^r$ and
  at time $\bar t$, we need to consider
  \begin{equation*}
    \mathcal{RS}^c \left((\rho^l, w^l),
      \left(\rho^r, w^r\right), u^-\right).
  \end{equation*}
  We have two different cases.
  \begin{enumerate}
  \item $\rho^r v^r \leq \varphi_{w^l, u^-}\left(\rho^r\right)$.
    In this case, the first family wave crosses the AV trajectory and
    no new wave is created.
    Thus, for the
    functionals~(\ref{eq:funct_F_w})-(\ref{eq:numbers-wave-func}), we have: 
    \begin{equation*}
      \Delta \mathcal F_w (\bar t) = \Delta \mathcal F_{\tilde v} (\bar t)
      = \Delta \mathcal F (\bar t)
      = \Delta \mathcal N (\bar t)
      = 0\,.
    \end{equation*}
    Here $\xi^- = \min\{v^l, u^-\}$, $\xi^+ = \min\{v^r, u^-\}$, and so
    $\modulo{\Delta \xi} \le \modulo{v^l - v^r}$.

  \item $\rho^r v^r > \varphi_{w^l, u^-}\left(\rho^r\right)$.
    In this case, the interaction
    produces a first family wave
    $((\rho^l, w^l), (\hat{\rho}, {w^l}))$, a \nfw wave
    $((\hat{\rho},{w^l}),(\check{\rho},{w^l}))$,
    and a phase transition wave $((\check{\rho}, w^l), (\rho^r, w^l))$,
    with positive speed.
    Thus, for the
    functionals~(\ref{eq:funct_F_w})-(\ref{eq:numbers-wave-func}),
    since $\check v > v^r > \hat v > v^l$, we have: 
    \begin{align*}
      \Delta \mathcal F_w (\bar t)
      & = \modulo{w^l - w^r} - \modulo{w^l - w^r} = 0\,,
      \\
      \Delta \mathcal F_{\tilde v} (\bar t)
      & = \modulo{v^l - \hat{v}} + \modulo{\check{v}-\hat{v}}
        +\modulo{\check{v}-v^r} - 2 \modulo{\check{v}-\hat{v}} -\modulo{v^l-v^r}
      \\
      & = \left(\hat v - v^l\right) - \left(\check v - \hat v\right)
        + \left(\check v - v^r\right) - \left(v^r - v^l\right)
      \\
      & = 2 {\left(\hat v - v^r\right)} < 0\,,
      \\
      \Delta \mathcal N (\bar t)
      & = 2\,.
    \end{align*}
    Here $\xi^- = \min\{v^l, u^-\}$, $\xi^+ = u^-$. Since $v^l < u^- < v^r$,
    then
    $\modulo{\Delta \xi} = \modulo{v^l - u^-} = u^- - v^l \le v^r - v^l$.
  \end{enumerate}
\end{proof}

\begin{remark}
  We observe that the following interaction
  can not happen:
\begin{itemize}
\item A \fw wave can not interact from the left with a second family
  wave. Indeed the velocity of cars coincides with that of the second family
  and so equation~\eqref{eq:Control-ODE} prevents such interaction.
\end{itemize}
\end{remark}

\subsubsection{Collisions from the left with a non fictitious wave}

In this part, we assume that a wave $((\rho^l, w^l), (\rho^r, w^r))$
interacts from the left at a time $\bar t$ with the AV and that there is a discontinuity
at the position of the AV before time $\bar t$.
For simplicity, we introduce also the following notations:
\begin{equation}
  \label{eq:simplicity-2}
  \begin{array}{r@{\,=\,}l@{\quad}r@{\,=\,}l@{\quad}r@{\,=\,}l@{\quad}
    r@{\,=\,}l}
    u^-
    & u\left(\bar t\right),
    & u^+
    & u^-,
    &     v^l
    & \tilde v(\rho^l, w^l),
    & v^r
    & \tilde v(\rho^r, w^r),
    \vspace{.2cm}\\
    \check{\rho}^l
    & \check{\rho}(u^-,w^l),
    & \check{\rho}^r
    & \check{\rho}(u^-,w^r),
    & \hat{\rho}^l
    & \hat{\rho}(u^-,w^l),
    & \hat{\rho}^r
    & \hat{\rho}(u^-,w^r),
    \vspace{.2cm}\\
    \check{v}^l
    & \tilde{v}(\check{\rho}^l,w^l),
    & \check{v}^r
    & \tilde{v}(\check{\rho}^r,w^r),
    & \hat{v}^l
    & \hat{\rho}(\hat{\rho}^l,w^l),
    & \hat{v}^r
    & \hat{\rho}(\hat{\rho}^r,w^r).
  \end{array}
\end{equation}
and we denote with $\xi^-$ and $\xi^+$ respectively the speed of the
AV before and after the interaction (see~\eqref{eq:velocity-AV-wft}),
and define $\Delta \xi = \xi^+ - \xi^-$.
 
\begin{lemma}
  \label{NY1}
  Assume that the second family wave
  $\left( (\rho^l, w^l), (\rho^r, w^r)\right)$ interacts from the left with the
  \nfw wave at the point
  $(\bar t,\bar x)$ with $\bar t > 0$ and $\bar x \in \reali$.
  The interaction is \textbf{$2$-\nfw/$1$-\nfw-$\mathcal{LW}$}.
    Moreover $\Delta \mathcal F_w (\bar t)=0$,
    $\Delta \mathcal F_{\tilde v} (\bar t)
    \le 2 \left(C_\psi w^l L_F \frac{1}{w_{\min} \bar \lambda}
          + {w^l C_\psi L_F}
          + 1\right) \modulo{w^l - w^r}$,
    $\Delta \mathcal F (\bar t)=\Delta \mathcal F_{\tilde v} (\bar t)$, and
    $\Delta \mathcal N (\bar t) \le \frac{V_{\max}}{\nu}$,
    where $\tilde C_F$ is a suitable positive constant depending only
    on~\ref{Hyp:H4}. Finally $\Delta \xi = 0$.
  \end{lemma}

\begin{proof}
  We use the notation in~(\ref{eq:simplicity-2}).
  \revv{Note that in this situation $u^- < V_{\max}$, otherwise
    the interaction wave can not happen.}
  The left and right states at the position of the AV before $\bar t$
  are given respectively by
  \begin{equation*}
    \left(\rho^r, w^r\right) = \left(\hat{\rho}^r, w^r\right)
    \qquad \textrm{ and } \qquad
    \left(\check{\rho}^r, w^r\right).
  \end{equation*}
  At time $\bar t$, we need to consider the Riemann solver
  \begin{equation*}
    \mathcal{RS}^c \left((\rho^l, w^l),
      \left(\check{\rho}^r, w^r\right), u^-\right).
  \end{equation*}
  Let $(\rho^m, w^m)$ be in the intersection between the free and the
  congested phase, with $w^m= w^l$.
  We have that $\rho^m v^m > \varphi_{w^m, u^-} (\rho^m)$. In this case, there is
    a production of a first family rarefaction wave
    $((\rho^l, w^l), ({\hat{\rho}^{l}}, w^l))$,
    of a \nfw wave $(({\hat{\rho}^{l}}, w^l), ({\check{\rho}^{l}}, w^l))$
    and of a linear wave connecting $({\check{\rho}^{l}}, w^l)$
    with $(\check{\rho}^{r}, {w^r})$.
    For the functionals~(\ref{eq:funct_F_w})-(\ref{eq:numbers-wave-func}),
    since $\check{v}^l> \hat{v}^l$, $ \check{v}^r >v^r$ and $v^l= v^r$, we have
    \begin{align*}
      \Delta \mathcal F_w (\bar t)
      & = 0\,,
      \\
      \Delta \mathcal F_{\tilde v} (\bar t)
      & = \modulo{v^l - \hat{v}^l} + \modulo{\hat{v}^l - \check{v}^l}
        + \modulo{\check{v}^l - \check{v}^r}
        - 2 \modulo{\check{v}^l - \hat{v}^l}
      \\
      & \quad
        - \modulo{v^l - v^r}
        - \modulo{v^r - \check{v}^r}
        + 2 \modulo{v^r - \check{v}^r}
      \\
      & = \modulo{v^l - \hat{v}^l} - \left(\check{v}^l - \hat{v}^l\right)
        + \modulo{\check{v}^l - \check{v}^r}
      \\
      & \quad + \left(\check{v}^r - v^l\right),
       \\
      \Delta \mathcal N(\bar t)
      & \le \frac{V_{\max}}{\nu}.
    \end{align*}
 
    We have that
      \begin{align*}
        \Delta \mathcal F_{\tilde v} (\bar t)
        & = \modulo{v^l - \hat{v}^l} - \left(\check{v}^l - \hat{v}^l\right)
          + \modulo{\check{v}^l - \check{v}^r} + \left(\check{v}^r - v^l\right)
        \\
        & = \modulo{v^l - \hat{v}^l} + \left(\hat{v}^l - v^l\right)
          + \modulo{\check{v}^l - \check{v}^r} + \left(\check{v}^r -
          \check v^l\right)
      \end{align*}
      and so
      \begin{equation*}
        \Delta \mathcal F_{\tilde v} (\bar t) \le 2
        \modulo{v^l - \hat{v}^l} + 2 \modulo{\check{v}^l - \check{v}^r}.
      \end{equation*}
      By~\ref{Hyp:H2}, \ref{Hyp:H4}, ~\eqref{eq:lip-fhi-check} and
      \Cref{tec_lem}, there exists $\tilde C_F > 0$ such that
      \begin{align*}
        \modulo{v^l - \hat{v}^l}
        & = w^l \modulo{\psi\left(\rho^l\right) - \psi\left(\hat \rho^l\right)}
          \le C_\psi w^l \modulo{\rho^l - \hat \rho^l}
        \\
        & \le C_\psi w^l \frac{1}{w^l \bar \lambda + u^-} \modulo{F_\alpha\left(w^l, u^-\right)
          - F_\alpha\left(w^r, u^-\right)}
        \\
        & \le C_\psi w^l L_F \frac{1}{w_{\min} \bar \lambda}
          \modulo{w^l - w^r}.
      \end{align*}
      Moreover, by~\ref{Hyp:H2}, \ref{Hyp:H4}, \eqref{eq:check-rho},
      and~\eqref{eq:lip-fhi-check},
      \begin{align*}
        \modulo{\check{v}^l - \check{v}^r}
        & = \modulo{w^l \psi\left(\check{\rho}^l\right) - w^r
          \psi\left(\check{\rho}^r\right)}
        \\
        & \le w^l C_\psi \modulo{\check{\rho}^l - \check{\rho}^r}
          + \psi\left(\check{\rho}^r\right) \modulo{w^l - w^r}
        \\
        & = w^l C_\psi \frac{\modulo{F_\alpha\left(w^l, u^-\right)
          - F_\alpha\left(w^r, u^-\right)}}{V_{\max} - u^-}
          + \psi\left(\check{\rho}^r\right) \modulo{w^l - w^r}
        \\
        & \le \left({w^l C_\psi L_F } + 1\right)
          \modulo{w^l - w^r}.
      \end{align*}
      Therefore we conclude that
      \begin{align*}
        \Delta \mathcal F_{\tilde v} (\bar t)
        & \le 2
          \modulo{v^l - \hat{v}^l} + 2 \modulo{\check{v}^l - \check{v}^r}
        \\
        &
          \le 2 \left(C_\psi w^l L_F \frac{1}{w_{\min} \bar \lambda}
          + {w^l C_\psi L_F}
          + 1\right) \modulo{w^l - w^r}.
      \end{align*}
      Here $\xi^- = \xi^+ = u^-$ and so $\Delta \xi = 0$.
    This completes the proof.
\end{proof}
  
\begin{lemma}
  \label{NY2}
  Assume that the phase transition wave $( (\rho^l, w^l), (\rho^r, w^r))$, with positive speed, interacts from the left with the \nfw wave at the point $(\bar t,\bar x)$ with $\bar t > 0$ and $\bar x \in \reali$. The only interaction is \textbf{$\mathcal{PT}$-\nfw/\fw-$\mathcal{LW}$}. Then $\Delta \mathcal F_w (\bar t)= \rev{\Delta}\mathcal F_{\tilde v} (\bar t) =\Delta \mathcal F(\bar t)=0$ and $\Delta \mathcal N(\bar t) = -1$.
  Finally $\Delta \xi = 0$.
\end{lemma}

\begin{proof}
 We use the notation in~(\ref{eq:simplicity-2}). Note that in this case $w^l=w^r$ and $\check{\rho}^l=\check{\rho}^r,\hat{\rho}^l=\hat{\rho}^r,\check{v}^l=\check{v}^r,\hat{v}^l=\hat{v}^r$ . At time $\bar t$, we need to consider the Riemann solver
  \begin{equation*}
    \mathcal{RS}^c \left((\rho^l, w^l),
      (\check{\rho}^l, w^l), u^-\right).
  \end{equation*}
In this case $\rho^l v(\rho^l,w^l) \leq \varphi_{w^l, u^-} (\rho^l)$. Thus, we apply the classical Riemann Problem between the states $(\rho^l, w^l)$ and $(\check{\rho}^l,\check{w})$, see~\cite[Theorem 2.1]{ColomboMarcelliniRascle}. That is, the phase transition
crosses the AV producing a linear wave $((\rho^l, w^l), (\check{\rho}^l,\check{w}))$. For the functionals~(\ref{eq:funct_F_w})-(\ref{eq:numbers-wave-func}), since $v^r < \check{v}^l < v^l$, we have
\begin{align*}
  \Delta \mathcal F_w (\bar t)
  & = 0\,,
  \\
  \Delta \mathcal F_{\tilde v} (\bar t)
  & = \modulo{v^l - \check{v}^l} - \modulo{v^l - v^r}
    - \modulo{v^r - \check{v}^l}
    + 2 \modulo{v^r - \check{v}^l}
  \\
  & = v^l - \check{v}^l - v^l + v^r
    + \check{v}^l- v^r
    = 0,
  \\
  \Delta \mathcal N(\bar t)
  & = -1.
\end{align*}
Here $\xi^- = u^-$, $\xi^+ = \min\{V_{\max}, u^-\} = u^-$, and so
$\Delta \xi = 0$.
\end{proof}
 
\begin{remark}
We note that the following interaction can not happen:
\begin{itemize}
\item A linear wave can not interact with a \nfw wave
  $((\hat{\rho}, \hat{w}), (\check{\rho}, \check{w}))$ from the left,
  since $(\hat{\rho}, \hat{w})$ is in the congested phase and a linear
  wave connects two states in the free phase.
\end{itemize}
\end{remark} 

 \subsubsection{Collisions from the right with a non fictitious wave}
 
 In this part, we assume that a wave interacts from the right at a time $\bar t$ with the AV and that there is a discontinuity at the position of the AV before time $\bar t$.
 Again denote with $\xi^-$ and $\xi^+$ respectively the speed of the
AV before and after the interaction (see~\eqref{eq:velocity-AV-wft}),
and define $\Delta \xi = \xi^+ - \xi^-$.
 
 \begin{lemma}
  \label{NY3}
  Assume that the \nfw $((\hat{\rho}, \hat{w}), (\check{\rho}, \check{w}))$
  interacts with the phase transition wave $((\check{\rho}, \check{w}), (\rho^r, w^r))$ at the point $(\bar t,\bar x)$ with $\bar t > 0$ and $\bar x \in \reali$. The only interaction is \textbf{\nfw-$\mathcal{PT}$/$1$-\fw}. Then $\Delta \mathcal F_w (\bar t)= \rev{\Delta}\mathcal F_{\tilde v} (\bar t) =\Delta \mathcal F(\bar t)=0$ and $\Delta \mathcal N(\bar t) = -1$.
  Finally $\modulo{\Delta \xi} \le v^l - v^r$.
\end{lemma}

\begin{proof}
 We use the notation in~(\ref{eq:simplicity-1}). At time $\bar t$, we need to consider the Riemann solver
  \begin{equation*}
    \mathcal{RS}^c \left((\hat{\rho}, \hat{w}),
      (\rho^r, w^r), u^-\right).
  \end{equation*}
In this case $\rho^r v(\rho^r,w^r) \leq \varphi_{w^r, u^-} (\rho^r)$. Thus, we apply the classical Riemann Problem between the states $(\hat{\rho}, \hat{w})$ and $(\rho^r, w^r)$, see~\cite[Theorem 2.1]{ColomboMarcelliniRascle}. That is, the phase transition wave crosses the AV producing a first family shock wave $((\hat{\rho}, \hat{w}), (\rho^r, w^r))$. For the functionals~(\ref{eq:funct_F_w})-(\ref{eq:numbers-wave-func}), since $\hat{v}> v^r$, $\check{v}> \hat{v}$ and $\check{v}> v^r$, we have
 \begin{align*}
      \Delta \mathcal F_w (\bar t)
      & = 0\,,
      \\
      \Delta \mathcal F_{\tilde v} (\bar t)
      & = \modulo{\hat{v} - v^r} - \modulo{\hat{v} - \check{v}}
        - \modulo{\check{v} - v^r}
        + 2 \modulo{\check{v} - \hat{v}}
        \\
        & = \hat{v} - v^r +\check{v}- \hat{v}- \check{v} + v^r
       = 0,
       \\
      \Delta \mathcal N(\bar t)
      & = -1.
    \end{align*}
    Here $\xi^- = u^-$, $\xi^+ = \min\{v^r, u^-\}$. If $u^- \le v^r$, then
    $\Delta \xi = 0$. If $u^- > v^r$, then
    $\modulo{\Delta \xi} = u^- - v^r \le V_{\max} - v^r \le v^l - v^r$.
 \end{proof}

\begin{remark}
We note that the following interactions can not happen:
\begin{itemize}
\item The \nfw wave $((\hat{\rho}, \hat{w}), (\check{\rho}, \check{w}))$
  can not interact from the left with a second family wave.
  Indeed if $(\check{\rho}, \check{w}) \in F \setminus C$, then the
  conclusion easily follows.
  If $(\check{\rho}, \check{w}) \in F \cap C$, then a wave of the second
  family coincides with a linear wave.
\item The \nfw wave $((\hat{\rho}, \hat{w}), (\check{\rho}, \check{w}))$
  can not interact from the left with a first family wave.  
  \end{itemize}
\end{remark} 
 
  \subsubsection{Collision with a special non fictitious wave}
  In this part we focus on the possible
  interactions of a \snfw with other waves.

  \begin{lemma}
    \label{le:snfw}
    Assume that, at time $\bar t > 0$, the \snfw connecting
    $\left(\rho^l, w^l\right)$ with $\left(\rho^r, w^r\right)$ interacts
    with another wave.
    Suppose that $w^l = w^r$, $(\rho^l, w^l) \in F \cap C$,
    $(\rho^r, w^l) \in F \setminus C$ and such that
    $\rho^r = \check \rho(w^l, 0)$.

    \begin{enumerate}
    \item The interaction with a first family wave is not possible.

    \item The interaction with a second family wave is not possible.

    \item The interaction with a linear wave is not possible.

    \item The interaction with a phase transition wave generates a (shock)
      wave of the first family and a \fw. More precisely the interaction is
      \textbf{\snfw-$\mathcal{PT}$/$1$-\fw}. In this case
      $\Delta \mathcal F_w (\bar t) = \Delta\mathcal F_{\tilde v} (\bar t)
      =\Delta \mathcal F(\bar t)=0$ and
      $\Delta \mathcal N(\bar t) = -1$.
      Finally $\modulo{\Delta \xi} = v(\rho^l, w^l) - v(\bar \rho^r, w^l)$,
      where $(\bar \rho^r, w^l)$ is the right state of the $\mathcal{PT}$.
    \end{enumerate}
  \end{lemma}

  \begin{proof}
    By assumption the \snfw is also a linear wave, so that its velocity
    is equal to $V_{\max}$. This implies that it can not interact with
    another linear wave or with a second family wave.

    In principle the \snfw can interact with a first family wave coming from
    the right. In this situation the state $\left(\rho^r, w^r\right)$
    is the left state of the wave of the first family, but, by
    hypothesis, $\left(\rho^r, w^r\right) \in F\setminus C$. This is not
    possible.

    Consider the case of the interaction with a phase transition wave
    connecting $\left(\rho^r, w^r\right) \in F \setminus C$ with
    $\left(\bar \rho^r, \bar w^r\right) \in C$. Clearly
    $\bar w^r = w^r = w^l$.
    At time $\bar t$, we need to consider the Riemann solver
    \begin{equation*}
      \mathcal{RS}^c \left((\rho^l, w^l),
        (\bar \rho^r, w^l), V_{\max} \right).
    \end{equation*}
    After the interaction, there is a (shock) wave of the first family,
    connecting $\left(\rho^l, w^l\right)$ with
    $\left(\bar \rho^r, w^l\right)$ and a \fw traveling at speed
    $\min \left\{V_{\max}, v(\bar \rho^r, w^l)\right\} = v(\bar \rho^r, w^l)$,
    so that the interaction is \textbf{\snfw-$\mathcal{PT}$/$1$-\fw}.
    
    For the functionals~(\ref{eq:funct_F_w})-(\ref{eq:numbers-wave-func}),
    since $\tilde v(\bar \rho^r, w^l) < \tilde v(\rho^l, w^l) = V_{\max}
    < \tilde v(\rho^r, w^l)$,
    we have
    \begin{align*}
      \Delta \mathcal F_w (\bar t)
      & = 0\,,
      \\
      \Delta \mathcal F_{\tilde v} (\bar t)
      & = \modulo{\tilde v(\rho^l, w^l) - \tilde v(\bar \rho^r, w^l)}
        + \modulo{\tilde v(\rho^l, w^l) - \tilde v(\rho^r, w^l)}
        - \modulo{\tilde v(\rho^r, w^l) - \tilde v(\bar \rho^r, w^l)}
      \\
      & = V_{\max} - \tilde v(\bar \rho^r, w^l)
        + \tilde v(\rho^r, w^l) - V_{\max}
        - \tilde v(\rho^r, w^l) + \tilde v(\bar \rho^r, w^l) = 0,
       \\
      \Delta \mathcal N(\bar t)
      & = -1.
    \end{align*}
    Here $\xi^- = V_{\max}$, $\xi^+ = v(\bar \rho^r, w^l)$.
    Thus $\modulo{\Delta \xi} = v(\rho^l, w^l) - v(\bar \rho^r, w^l)$.
  \end{proof}

We collect all the interaction estimates between two waves
in~\Cref{tab:interaction-estimates}. 
\begin{table}
  \centering
  \begin{tabular}{|c|c|c|c|l|}
    \hline
    Interaction type
    & $\Delta \mathcal F_w$
    & $\Delta \mathcal F_v$
    & $\Delta \mathcal N$
    & Lemma
    \\
    \hline \hline
    \scriptsize\textbf{$2$-$1$/$1$-$2$}
    &
      \multirow{4}{*}{$= 0$}
    &
      \multirow{4}{*}{$\le 0$}
    &
      \multirow{4}{*}{$\le 0$}
    & \multirow{4}{*}{\scriptsize \Cref{le:estim-funct-w-interaction-PT}}
    \\\scriptsize
    \textbf{$\mathcal{LW}$-$\mathcal{PT}$/$\mathcal{PT}$-$2$}
        &
    &
    &
    & 
    \\\scriptsize
    \textbf{$1$-$1$/$1$}
    &&&&
    \\\scriptsize
    \textbf{$\mathcal{PT}$-$1$/$\mathcal{PT}$}
    &&&&
    \\
    \hline\scriptsize
    \textbf{$2$-$\mathcal{FW}$/$\mathcal{FW}$-$2$}
    &
      \multirow{2}{*}{$=0$}
    &
      \multirow{2}{*}{$=0$}
    &
      $= 0$
    & \multirow{2}{*}{\scriptsize \Cref{Y1}}
    \\\scriptsize
    \textbf{$2$-$\mathcal{FW}$/$1$-$\mathcal{NFW}$-$\mathcal{PT}$-$2$}
    &&&
        $=3$
    &
    \\
    \hline\scriptsize
    \textbf{$\mathcal{PT}$-$\mathcal{FW}$/$\mathcal{FW}$-$\mathcal{PT}$}
    &
      $=0$
    & $=0$
    & $=0$
    &
      \scriptsize \Cref{Y2}
    \\
    \hline\scriptsize
    \textbf{$\mathcal{FW}$-$\mathcal{PT}$/$\mathcal{PT}$-$\mathcal{FW}$}
    &
      $=0$
    & $=0$
    & $=0$
    &
      \scriptsize \Cref{Y2Bis}
    \\
    \hline\scriptsize
    \textbf{$\mathcal{LW}$-\fw/$\mathcal{FW}$-$\mathcal{LW}$}
    &
      \multirow{2}{*}{$=0$}
    &
      $=0$
    & $=0$
    & \multirow{2}{*}{\scriptsize \Cref{Y3}}
    \\
    \scriptsize
    \textbf{$\mathcal{LW}$-\fw/$\mathcal{PT}$-\nfw-$\mathcal{LW}$}
    &&
       $\le 0$
    &
      $=2$
    &
    \\
    \hline\scriptsize
    \textbf{\fw-$1$/$1$-$\mathcal{FW}$}
    &
      \multirow{2}{*}{$=0$}
    & $=0$
    & $=0$
    & \multirow{2}{*}{\scriptsize \Cref{Y4}}
    \\\scriptsize
    \textbf{\fw-$1$/$1$-\nfw-$\mathcal{PT}$}
    &&
       $<0$
    & $=2$
    &
    \\
    \hline
    \scriptsize\textbf{$2$-\nfw/$1$-\nfw-$\mathcal{LW}$}
    & $=0$
    &
      \small$\le O(1) \modulo{w^l - w^r}$
    &
      $\le \frac{V_{\max}}{\nu}$
    &
      \scriptsize \Cref{NY1}
    \\
    \hline
    \scriptsize
    \textbf{$\mathcal{PT}$-\nfw/\fw-$\mathcal{LW}$}
    &
      $=0$
    &
      $=0$
    & $=-1$
    &
      \scriptsize \Cref{NY2}
    \\
    \hline\scriptsize
    \textbf{\nfw-$\mathcal{PT}$/$1$-\fw}
    &
      $=0$
    &
      $=0$
    & $=-1$
    &
      \scriptsize \Cref{NY3}
    \\
    \hline
     \scriptsize
     \revv{\textbf{\snfw-$\mathcal{PT}$/$1$-\fw}}
    &
       \revv{$=0$}
    &
       \revv{$=0$}
    &  \revv{ $=-1$}
    &
       \revv{\scriptsize \Cref{le:snfw}}
    \\
    \hline
  \end{tabular}
  \caption{The variation of the functionals $\mathcal F_w$,
    $\mathcal F_v$, and $\mathcal N$ due to interactions between waves.
                                                    The Landau symbol $O(1)$ denotes a constant; see \Cref{NY1} for the precise
    expression.}
  \label{tab:interaction-estimates}
\end{table}

\subsubsection{Control changes}
We focus here on the situations in which a jump in the control function $u$ occurs.
and denote with $\xi^-$ and $\xi^+$ respectively the speed of the
AV before and after the interaction (see~\eqref{eq:velocity-AV-wft}),
and define $\Delta \xi = \xi^+ - \xi^-$.
\begin{lemma}
  \label{Control1}
  Assume that, at time $\bar t > 0$, the control function $u$ jumps from $u^{-}=u(\bar t -)$ to $u^{+}=u(\bar t +)$ and that we have a \fw at time $\bar  t-$. We have the following cases.
  \begin{enumerate}
  \item  At time $\bar t+$ we have a \fw and no new wave is produced. Then $\Delta \mathcal F_w (\bar t)= \rev{\Delta}\mathcal F_{\tilde v} (\bar t) =\Delta \mathcal F(\bar t)=0$ and $\Delta \mathcal N(\bar t) = 0$.
    Finally $\modulo{\Delta \xi} \le \modulo{u^- - u^+}$.
    
\item At time $\bar t+$ we have a \nfw and the number of waves increases.
  Then $(\rho^{l},w^{l}) \in C$, \textbf{\fw /$1$-\nfw -$\mathcal{PT}$},
  and $u^+ < u^-$.
  The phase transition wave coincides with a linear wave in the case
  $(\rho^{l},w^{l}) \in F \cap C$.
  Moreover $\Delta \mathcal F_w (\bar t)= \rev{\Delta}\mathcal F_{\tilde v} (\bar t) =\Delta \mathcal F(\bar t)=0$ and $\Delta \mathcal N(\bar t) = 3$.
  Finally $\modulo{\Delta \xi} \le \modulo{u^- - u^+}$.
 \end{enumerate}
\end{lemma}

\begin{proof}
  Here we use the following notations:
  \begin{equation*}
    \begin{array}{r@{\,=\,}l@{\quad}r@{\,=\,}l@{\quad}r@{\,=\,}l@{\quad}
      r@{\,=\,}l}
      u^-
      & u\left(\bar t-\right),
      & u^+
      & u\left(\bar t+\right),
      &     v^l
      & \tilde v(\rho^l, w^l),
      & v^r
      & \tilde v(\rho^r, w^r),
        \vspace{.2cm}\\
      \check{\rho}^l
      & \check{\rho}(u^-,w^l),
      & \check{\rho}^r
      & \check{\rho}(u^-,w^r),
      & \hat{\rho}^l
      & \hat{\rho}(u^-,w^l),
      & \hat{\rho}^r
      & \hat{\rho}(u^-,w^r),
        \vspace{.2cm}\\
      \check{v}^l
      & \tilde{v}(\check{\rho}^l,w^l),
      & \check{v}^r
      & \tilde{v}(\check{\rho}^r,w^r),
      & \hat{v}^l
      & \hat{\rho}(\hat{\rho}^l,w^l),
      & \hat{v}^r
      & \hat{\rho}(\hat{\rho}^r,w^r),
    \end{array}
  \end{equation*}
  At time $\bar t$, we need to consider the Riemann solver
  \begin{equation*}
    \mathcal{RS}^c \left((\rho^l, w^l),
      (\rho^l, w^l), u^+ \right).
  \end{equation*}
We have two different cases.  
 \begin{enumerate}
\item $\rho^l v^l \leq \varphi_{w^l, u^+}(\rho^l)$.
  In this case, at time $\bar t+$, no new wave is produced and for the
  functionals~(\ref{eq:funct_F_w})-(\ref{eq:numbers-wave-func}), we have
  \begin{equation*}
    \Delta \mathcal F_w (\bar t)= \mathcal F_{\tilde v} (\bar t) =\Delta \mathcal
    F(\bar t)=\Delta \mathcal N(\bar t) = 0.
  \end{equation*}
  Here $\xi^- = \min\{v^l, u^-\}$, $\xi^+ = \min\{v^l, u^+\}$, and so
  $\modulo{\Delta \xi} \le \modulo{u^- - u^+}$.
  
\item $\rho^l v^l > \varphi_{w^l, u^+}(\rho^l)$.
  In this case $(\rho^{l},w^{l}) \in C$, $u^+ < u^-$, and there is a production of a first family wave $((\rho^l, w^l), ({\hat{\rho}}, w^l))$, of a \nfw wave $(({\hat{\rho}}, w^l), ({\check{\rho}}, w^l))$ and of a phase transition wave connecting $({\check{\rho}}, w^l)$ with $(\rho^l, {w^l})$.
  Note that, if $(\rho^l, w^l) \in F \cap C$,
  then the phase transition
  wave is indeed a liner wave.
For the functionals~(\ref{eq:funct_F_w})-(\ref{eq:numbers-wave-func}), since $\check{v} > v^l > \hat{v}$, we have
\begin{align*}
      \Delta \mathcal F_w (\bar t)
      & = 0\,,
      \\
      \Delta \mathcal F_{\tilde v} (\bar t)
      & = \modulo{v^l - \hat{v}} + \modulo{\hat{v} - \check{v}}
        + \modulo{\check{v} - v^l}
        - 2 \modulo{\check{v} - \hat{v}}
        \\
        & = v^l - \hat{v} -\check{v}+ \hat{v}+ \check{v} - v^l
       = 0,
       \\
      \Delta \mathcal N(\bar t)
      & = 3.
\end{align*}
Here $\xi^- = \min\{v^l, u^-\}$, $\xi^+ = u^+$, and $v^l > u^+$.
Thus $\modulo{\Delta \xi} \le \modulo{u^- - u^+}$.
\end{enumerate}
\end{proof}

\begin{lemma}
  \label{Control2}
  Assume that, at time $\bar t > 0$, the control function $u$ jumps from
  $u^{-}=u(\bar t -)$ to $u^{+}=u(\bar t +)$ and that we have a \nfw
  at time $\bar  t-$, connecting $(\rho^l, w^l)$ with
  $(\rho^r, w^l)$. We have the following cases.
  \begin{enumerate}
  \item \revv{At time $\bar t+$ we have a \snfw and the number
      of waves increases.
      The production is \textbf{\nfw /$1$-\snfw ($\mathcal{LW}$)}.
      In this case, the left state of the \snfw belongs to $F \cap C$, the
      right state belongs to $F \setminus C$, and they have the same $w$.
      Moreover $\Delta \mathcal F_w (\bar t)=0$, $\Delta\mathcal F_{\tilde v}
      (\bar t) \le
      2 w^l C_\psi \frac{F_{\alpha,1}(w_{\max}) + R}{w_{\min} \bar \lambda}
      (u^+ - u^-),$ and
      $\Delta \mathcal N(\bar t) \le 1 + \frac{V_{\max}}{\nu}$.
      Finally $\modulo{\Delta \xi} \le \modulo{u^- - u^+}$.}

  \item At time $\bar t+$ we have a \nfw and the number of waves
    increases. The production is \textbf{\nfw /$1$-\nfw
      -$\mathcal{LW}$}. Moreover $\Delta \mathcal F_w (\bar t)=0$,
    $\Delta\mathcal F_{\tilde v} (\bar t) \le 2 w^l C_\psi
    \frac{F_{\alpha, 1}(w^l) + R}{w^l \bar \lambda} \modulo{u^+ -
      u^-}$ and
    $\Delta \mathcal N(\bar t) \le 1 + \frac{V_{\max}}{\nu}$.
    Finally it holds
    $\modulo{\Delta \xi} \le \modulo{u^- - u^+}$.
  \end{enumerate}
\end{lemma}

\begin{proof}
  We use the following notations:
  \begin{equation*}
    \begin{array}{r@{\,=\,}l@{\quad}r@{\,=\,}l@{\quad}r@{\,=\,}l@{\quad}
      r@{\,=\,}l}
      u^-
      & u\left(\bar t-\right),
      & u^+
      & u\left(\bar t+\right),
      & \check \rho
      & \check \rho ( u^-,w^l),
      & \hat \rho
      & \hat \rho ( u^-,w^l),
        \vspace{.2cm}\\
      v^l
      & \tilde v(\rho^l, w^l),
      & v^r
      & \tilde v(\rho^r, w^l),
      & \check{v}
      & \tilde v(\check{\rho}, {w^l}),
      & \hat{v}
      & \tilde v(\hat{\rho}, {w^l}),
        \vspace{.2cm}\\
      \check{\rho}^+
      & \check{\rho}(u^+,w^l),
      & \hat{\rho}^+
      & \hat{\rho}(u^+,w^l),
      & \check{v}^+
      & \check{v}(u^+,w^l),
      & \hat{v}^+
      & \hat{v}(u^+,w^l).
    \end{array}
  \end{equation*}
  \revv{Here in particular we have that $v^l = \hat v$ and $v^r = \check v$.}
  At time $\bar t$, we need to consider the Riemann solver
  \begin{equation*}
    \mathcal{RS}^c \left((\hat{\rho}, w^l),
      (\check{\rho}, w^l), u^+ \right).
  \end{equation*}
  Let $(\rho^m, w^m)$ be in the intersection between the free
  and the congested phase, with $w^m= w^l$. We have two different cases. 
 \begin{enumerate}
 \item $\rho^m v^m \leq \varphi_{w^l, u^+}(\rho^m)$.
   By~\eqref{eq:check-rho}, $\check \rho(w^l, \sigma)$ depends only on $w^l$
   for every $\sigma \in [0, V_{\max}]$.
   Moreover $\rho^m \ge \check \rho(w^l, \sigma)$
   for $\sigma \in [0, V_{\max}]$.
   Therefore we deduce that $u^+ = V_{\max}$
   \revv{and $\rho^m v^m = \varphi_{w^l, u^+}(\rho^m)$}.
   \revv{In this case there is a production of a first family wave
     $((\hat{\rho}, w^l), (\rho^m, w^l))$ and of a linear wave connecting
     $(\rho^m, w^l)$ with $(\check{\rho}, {w^l})$, which is also
     a $\mathcal{SNFW}$ wave.
     Note that the left state of the \snfw belongs to $F \cap C$, while the
     right state belongs to $F \setminus C$.
     For the functionals~(\ref{eq:funct_F_w})-(\ref{eq:numbers-wave-func}),
     using~\ref{Hyp:H2} and~\Cref{tec_lem},
     we have
     \begin{align*}
       \Delta \mathcal F_w (\bar t)
      & = 0\,,
    \\
    \Delta \mathcal F_{\tilde v} (\bar t)
      & = \modulo{\hat{v}- v^m} + \modulo{{v}^{m} - \check{v}}
        - 2 \modulo{{v}^m - \check{v}}
    \\
      & - \modulo{\hat{v} - \check{v}}+ 2 \modulo{\hat{v} - \check{v}}
    \\
      & = (v^m -\hat v) -(\check v - v^m) + (\check v - \hat v)
    \\
       & = 2 (v^m - \hat v) = 2 (v^m - v^l)
       \\
       & = 2 w^l \left(\psi(\rho^m) - \psi(\rho^l)\right)
       \\
       & \le 2 w^l C_\psi \modulo{\rho^l - \rho^m},
       \\
       & \le 2 w^l C_\psi \frac{ F_{\alpha,1}(w_{\max}) + R}{w_{\min} \bar \lambda}
         (u^+ - u^-),
    \\
    \Delta \mathcal N(\bar t)
      & \le 1 + \frac{V_{\max}}{\nu}.
     \end{align*}
     Here $\xi^- = u^-$, $\xi^+ = u^+$, and so
     $\modulo{\Delta \xi} \le \modulo{u^- - u^+}$.
   }
    
\item $\rho^m v^m > \varphi_{w^l, u^+}(\rho^m)$.
  In this case there is a production of a first family wave $((\hat{\rho}, w^l), ({\hat{\rho}^{+}}, w^l))$, of a \nfw wave $(({\hat{\rho}^{+}}, w^l), ({\check{\rho}^{+}}, w^l))$ and of a linear wave connecting $({\check{\rho}^{+}}, w^l)$ with $(\check{\rho}, {w^l})$. For the functionals~(\ref{eq:funct_F_w})-(\ref{eq:numbers-wave-func}),
  since  $\check{v}^{+} > \hat{v}^{+}$ and $\check{v} > \hat{v}$,
  by~\Cref{le:lip-v-rho}, we have
  \begin{align*}
    \Delta \mathcal F_w (\bar t)
      & = 0\,,
    \\
    \Delta \mathcal F_{\tilde v} (\bar t)
      & = \modulo{\hat{v}- \hat{v}^{+}} + \modulo{\hat{v}^{+} - \check{v}^{+}}
        + \modulo{\check{v}^{+} - \check{v}}
        - 2 \modulo{\hat{v}^{+} - \check{v}^{+}}
    \\
      & - \modulo{\hat{v} - \check{v}}+ 2 \modulo{\hat{v} - \check{v}}
    \\
      & = \modulo{\hat{v}^{+} -\hat{v}} - \check{v}^{+} + \hat{v}^{+} + \modulo{\check{v}^{+} - \check{v}}+\check{v} - \hat{v}
    \\
      & \le 2 \modulo{\hat{v}^{+}  - \hat{v}}  + 2 \modulo{\check{v}^{+} - \check{v}}  
    \\
      & \le 2 w^l C_\psi \left(\modulo{\hat \rho^+ - \hat \rho} +
        \modulo{\check \rho^+ - \check \rho}\right),
    \\
    \Delta \mathcal N(\bar t)
      & \le 1 + \frac{V_{\max}}{\nu}.
  \end{align*}
  By~\eqref{eq:lip-fhi-check}, we have that
  $\modulo{\check \rho^+ - \check \rho} = 0$.

  By~\Cref{tec_lem}, 
  \begin{equation*}
    \modulo{\hat \rho^+ - \hat \rho} \le
    \frac{F_{\alpha, 1}(w^l) + R}{w^l \bar \lambda} \modulo{u^+ - u^-}.
  \end{equation*}
  Hence
  \begin{align*}
    \Delta \mathcal F_{\tilde v} (\bar t)
    & \le 2 w^l C_\psi \frac{F_{\alpha, 1}(w^l) + R}{w^l \bar \lambda}
      \modulo{u^+ - u^-}.
  \end{align*}
  Here $\xi^- = u^-$, $\xi^+ = u^+$, and so
  $\modulo{\Delta \xi} \le \modulo{u^- - u^+}$.
\end{enumerate}
This completes the proof.
\end{proof}

\begin{lemma}
  \label{Control3}
  \revv{Assume that, at time $\bar t > 0$, the control function $u$ jumps from
    $u^{-} = V_{\max}$ to $u^{+} = u(\bar t +)$ and that we have a \snfw
    at time $\bar  t-$, connecting $(\rho^l, w^l) \in F \cap C$ with
    $(\rho^r, w^l) \in F \setminus C$ and such that
    $\rho^r = \check \rho(w^l, \sigma)$ for every $\sigma \in [0, V_{\max}[$.}

  \revv{Then, at time $\bar t+$ we have a \nfw and the number
      of waves increases.
      The production is \textbf{\snfw /$1$-\nfw}.
      Moreover $\Delta \mathcal F_w (\bar t)=0$,
      $\Delta\mathcal F_{\tilde v} (\bar t) = 0$, and
      $\Delta \mathcal N(\bar t) = 1$.
      Finally $\modulo{\Delta \xi} = \modulo{u^- - u^+}$.}
\end{lemma}

\begin{proof}
  \revv{ We use the following notations:
    \begin{equation*}
      \begin{array}{r@{\,=\,}l@{\quad}r@{\,=\,}l@{\quad}r@{\,=\,}l@{\quad}
        r@{\,=\,}l}
        u^-
        & u\left(\bar t-\right),
        & u^+
        & u\left(\bar t+\right),
        & \check \rho
        & \check \rho (w^l,  u^+),
        & \hat \rho
        & \hat \rho (w^l,  u^+),
          \vspace{.2cm}\\
        v^l
        & \tilde v(\rho^l, w^l),
        & v^r
        & \tilde v(\rho^r, w^l),
        & \check v
        & \tilde v(\check \rho, w^l),
        & \hat{v}
        & \tilde v(\hat{\rho}, {w^l}).
          \vspace{.2cm}\\
      \end{array}
    \end{equation*}
    At time $\bar t$, we need to consider the Riemann solver
    \begin{equation*}
      \mathcal{RS}^c \left(({\rho^l}, w^l),
        ({\rho^r}, w^l), u^+ \right).
    \end{equation*}
    Since, by assumption, $\rho^r = \check \rho$, we deduce that there
    is a production of a first family (shock) wave
    $(({\rho^l}, w^l), (\hat \rho, w^l))$ and of a \nfw connecting
    $(\hat \rho, w^l)$ with $({\rho^r}, {w^l})$.  }

   \revv{
     For the functionals~(\ref{eq:funct_F_w})-(\ref{eq:numbers-wave-func}),
     since $\hat v < v^\ell = V_{\max} \le \check v = v^r$ and the
     wave of the first family is a shock,
     we have
     \begin{align*}
       \Delta \mathcal F_w (\bar t)
       & = 0\,,
       \\
       \Delta \mathcal F_{\tilde v} (\bar t)
       & = \modulo{\hat{v}- v^l} + \modulo{\hat{v} - {v^r}}
        - 2 \modulo{\hat {v} - {v^r}}
         - \modulo{{v^l} - {v^r}}+ 2 \modulo{{v^l} - {v^r}}
       \\
       & = \left(v^l - \hat{v}\right) - \left(v^r - \hat{v}\right)
         + \left(v^r - v^l\right) = 0
       \\
       \Delta \mathcal N(\bar t)
       & = 1.
     \end{align*}
     Here $\xi^- = u^- = V_{\max}$, $\xi^+ = u^+$, and so
     $\modulo{\Delta \xi} = \modulo{u^- - u^+}$, concluding the proof.
   }
\end{proof}

We collect all the estimates due to the change of the control
in~\Cref{tab:control-changes-estimates}.
\begin{table}
  \centering
  \begin{tabular}{|c|c|c|c|l|}
    \hline
    Waves' type
    & $\Delta \mathcal F_w$
    & $\Delta \mathcal F_v$
    & $\Delta \mathcal N$
    & Lemma
    \\
    \hline \hline
    \scriptsize    \textbf{$\mathcal{FW}$/$\mathcal{FW}$}
    &
      \multirow{2}{*}{$= 0$}
    &
      \multirow{2}{*}{$= 0$}
    &
      $= 0$
    & \multirow{2}{*}{\scriptsize\Cref{Control1}}
    \\\scriptsize
    \textbf{\textbf{\fw /$1$-\nfw -$\mathcal{PT}$}}
    &
    &
    &
      $=2$
    & 
    \\
    \hline\scriptsize
    \revv{\textbf{\textbf{\nfw /$1$--\snfw ($\mathcal{LW}$)}}}
    &
      \multirow{2}{*}{$=0$}
    &
      \multirow{2}{*}{$\le O(1) \modulo{u^+ - u^+}$}
    &
      \multirow{2}{*}{$\le 1+ \frac{V_{\max}}{\nu}$}
    & \multirow{2}{*}{\scriptsize\Cref{Control2}}
    \\
    \scriptsize
    \textbf{\textbf{\nfw /$1$-\nfw -$\mathcal{LW}$}}
    &&&&
    \\
    \hline
    \scriptsize
    \revv{\textbf{\snfw /$1$--\nfw}}
    &
      \revv{$= 0$}
    &
      \revv{$= 0$}
    &
      \revv{$= 1$}
    &
      \revv{\scriptsize\Cref{Control3}}
    \\
    \hline
  \end{tabular}
  \caption{The variation of the functionals $\mathcal F_w$,
    $\mathcal F_v$, and $\mathcal N$ due to control changes.
    The Landau symbol $O(1)$ denotes a constant; see \Cref{Control2}
    for the precise
    expression.}
  \label{tab:control-changes-estimates}
\end{table}

\subsection{Existence of a wave-front tracking approximate solution}
In this part we deal with the existence of a wave-front tracking approximate
solution, in the sense of \Cref{def:epswf}.

\begin{proposition}
  \label{prop:existence-wft}
  For every $\nu \in \naturali \setminus \left\{0\right\}$,
  the construction illustrated in~\Cref{sse:wft-12} produces a wave-front
  tracking approximate solution, defined for every time $t \ge 0$.
\end{proposition}
\begin{proof}
  Fix $\nu \in \naturali \setminus \left\{0\right\}$.
  We need to prove that the total number of waves and
  interactions remain finite.
  By construction $u_\nu$ is piecewise constant with a finite number of
  discontinuities.
  \revv{Note also that, due to \Cref{le:snfw} and to \Cref{Control3},
    the \snfw is superimposed only to a linear wave with left state in
    $F \cap C$ and right state in $F \setminus C$ with the same $w$ coordinate
    as described in \Cref{Control3}.}
  Hence the interactions, described in \Cref{Control1},
  in \Cref{Control2}, and in \Cref{Control3},
  can happen at most a finite number of times
  and \revv{so} the number of new waves,
  generated by the changes of the control,
  is also finite. Thus, without loss of generality, we may assume that
  the control $u_\nu$ is constant. \revv{In particular it is possible
    to assume that \snfw is not present.}

  New waves of the first family can not be generated to the right of the AV;
  see~\Cref{tab:interaction-estimates}. Therefore the functional
  $t \mapsto \mathcal N^+_1(t)$ does not increase and so, for $t \ge 0$,
  \begin{equation}
    \label{eq:N1+}
    \mathcal N_1^+(t) \le \mathcal N_1^+(0+),
  \end{equation}
  which implies that also the interactions, studied in \Cref{Y4},
  can happen at most a finite number of times.

  New $\mathcal{LW}$ waves can not be generated to the left of the AV;
  see~\Cref{tab:interaction-estimates}. Therefore the functional
  $t \mapsto \mathcal N^-_{\mathcal{LW}}(t)$ does not increase
  and so, for $t \ge 0$,
  \begin{equation}
    \label{eq:NLW-}
    \mathcal N_{\mathcal{LW}}^-(t) \le \mathcal N_{\mathcal{LW}}^-(0+),
  \end{equation}
  which implies that also the interactions, studied in \Cref{Y3},
  can happen at most a finite number of times.

  If, at time $t > 0$, there are two $\mathcal{PT}$ waves at the left (or
  at the right) of the AV, then there exists at least one $\mathcal{LW}$ and
  one wave of the first family in between. Thus, for $t > 0$,
  using~\eqref{eq:N1+} and~\eqref{eq:NLW-},
  \begin{align*}
    \mathcal N_{\mathcal{PT}}^-(t)
    & \le \mathcal N_{\mathcal{LW}}^-(t) + 1 \le
      \mathcal N_{\mathcal{LW}}^-(0+) + 1
    \\
    \mathcal N_{\mathcal{PT}}^+(t)
    & \le \mathcal N_{1}^+(t) + 1 \le
      \mathcal N_{1}^+(0+) + 1
  \end{align*}
  and so
  \begin{equation*}
    \mathcal N_{\mathcal{PT}}(t)
    = \mathcal N_{\mathcal{PT}}^-(t) + \mathcal N_{\mathcal{PT}}^+(t)
    \le \mathcal N_{\mathcal{LW}}^-(0+) + \mathcal N_{1}^+(0+) + 2.
  \end{equation*}

  We now focus on the number of interactions of type
  \textbf{$2$-\nfw/$1$-\nfw-$\mathcal{LW}$}, described in~\Cref{NY1}.
  New waves of the second family can be created at a positive time
  only with the interaction
  \textbf{$\mathcal{LW}$-$\mathcal{PT}$/$\mathcal{PT}$-$2$}.
  Then, the number of times the interactions
  \textbf{$2$-\nfw/$1$-\nfw-$\mathcal{LW}$}
  may happen is bounded by
  $\mathcal N_2^-(0+) + \mathcal N_{\mathcal{LW}}^-(0+)$, since~\eqref{eq:NLW-}
  holds.

  With the same reasoning, the interactions, described in~\Cref{Y1},
  can happen at most $\mathcal N_2^-(0+) + \mathcal N_{\mathcal{LW}}^-(0+)$
  times.

  Since the functional $\mathcal N$ strictly increases only in the interactions considered in
  \Cref{Y1}, \Cref{Y3}, \Cref{Y4} and of \Cref{NY1}, then
  \begin{equation}
    \label{eq:N-estimate}
    \begin{split}
      \mathcal N(t)
      & \le \mathcal N(0+)
      + 3 \left(\mathcal N_2^-(0+) + \mathcal N_{\mathcal{LW}}^-(0+)\right)
      + 2 \mathcal N_{\mathcal{LW}}^-(0+)
      \\
      & \quad + 2 \mathcal N_{1}^+(0+)
      + \frac{V_{\max}}{\nu}
      \left(\mathcal N_2^-(0+) + \mathcal N_{\mathcal{LW}}^-(0+)\right)
      \\
      & \le \mathcal N(0+) + 2 \mathcal N_{1}^+(0+)
      + \left(3 + \frac{V_{\max}}{\nu}\right)
      \mathcal N_2^-(0+)
      \\
      & \quad + \left(5 + \frac{V_{\max}}{\nu}\right)
      \mathcal N_{\mathcal{LW}}^-(0+)
      \\
      & \le \left(6 + \frac{V_{\max}}{\nu}\right)
      \mathcal N(0+).
    \end{split}
  \end{equation}
  Since in the interactions described in \Cref{NY2} and in \Cref{NY3}
  the number of waves strictly decreases and since~\eqref{eq:N-estimate} holds,
  then these interactions
  can happen at most a finite number of times.
  Similarly also the interactions
  \textbf{$1$-$1$/$1$} and \textbf{$\mathcal{PT}$-$1$/$\mathcal{PT}$}
  can happen at most a finite number of times.

  We claim now that the number of interactions
  \textbf{$2$-$1$/$1$-$2$} is finite.
  Indeed by~\eqref{eq:N1+} the number of interactions
  \textbf{$2$-$1$/$1$-$2$} at the right of AV is finite. Moreover,
  as already proved, the number of waves of the first family
  generated at a positive time at the AV location is finite;
  thus the number of interactions
  \textbf{$2$-$1$/$1$-$2$} at the left of AV is also finite, proving the
  claim.
  
  Symmetrically also the number of interactions
  \textbf{$\mathcal{LW}$-$\mathcal{PT}$/$\mathcal{PT}$-$2$} is finite.
  This is a consequence of the fact that the number of new
  linear waves generated at the location of the AV is finite.
  Therefore all the interactions described
  in~\Cref{le:estim-funct-w-interaction-PT} are finite.

  It remains to prove that the interactions in \Cref{Y2} and in \Cref{Y2Bis}
  can happen at most a finite number of times.
  Indeed when a phase-transition wave interacts with the fictitious wave,
  the waves cross each other and they can not interact anymore without other
  interactions before. This prevents the possibility of a combination
  of the two interactions happens an infinite number of times.
  
  The proof is so concluded.
\end{proof}

\subsection{Existence of a Solution}
This section deals with the proof of the main result.

\begin{proofof}{\Cref{thm:cauchy-problem}}
  Fix $\left(z_\eps, y_\eps, u_\eps\right)$, an $\varepsilon$-approximate wave-front tracking solution
   to~\eqref{eq:CP},
  in the sense of \Cref{def:epswf}, which exists by~\Cref{prop:existence-wft}.
  By assumptions there exists $M > 0$ such that
  \begin{equation*}
    \mathcal F_w(0+) \le M,
    \qquad
    \mathcal F_{v}(0+) \le M,
    \qquad \textrm{ and } \qquad
    \tv\left(u_\eps\right) \le M.
  \end{equation*}
  In particular, at least passing to a
  subsequence, there exists a control
  $u^* \in \BV\left([0, +\infty[; [0, \overline V]\right)$
  such that
  \begin{equation*}
    u_\eps(t) \to u^*(t)
  \end{equation*}
  for a.e. $t > 0$ as $\eps \to 0$.

  By the interaction estimates in \Cref{2},
  see also \Cref{tab:interaction-estimates}
  and \Cref{tab:control-changes-estimates},
  we deduce that
  \begin{equation*}
    \mathcal F_w(t) = \mathcal F_w(0+) \le M
  \end{equation*}
  for every $t > 0$.
  Moreover the possible increments of the functional $\mathcal F_v$ are
  described in \Cref{NY1} and in \revv{ the two cases of }~\Cref{Control2}.
  In particular, the maximum possible
  increment due to the control's change is proportional to the total variation
  of $u_\eps$; see~\Cref{Control2}.
  Instead the maximum possible increment due to the interactions described
  in \Cref{NY1} is proportional to $\mathcal F_w(0+)$,
  since the waves of the second family interacting with the AV either are
  original waves, i.e. generated at time $t=0$, or are generated by the
  interaction $\mathcal{LW}$-$\mathcal{PT}$/$\mathcal{PT}$-$2$ at the left
  of the AV. The $\mathcal{LW}$ waves can not be generated at positive
  time at the left of the AV and the strength $\mathcal F_v$ of $\mathcal{LW}$
  is in the previous interaction is transferred to the wave of the second
  family.
  This permits to conclude that
  \begin{equation*}
    \mathcal F_{\tilde v}(t) \le O(1) M
  \end{equation*}
  for every $t > 0$, where the Landau symbol $O(1)$ denotes a constant
  not depending on $t$, $\eps$, and $M$.
  
  Hence, at least passing to a
  subsequence, there exist $w^*$ and $\tilde v^*$ such that
  \begin{equation*}
    w\left(z_{\eps}\left(t, \cdot\right)\right) \to w^*(t, \cdot)
    \qquad \textrm{ and }\qquad
    \tilde v\left(z_{\eps}\left(t, \cdot\right)\right) \to
    \tilde v^*(t, \cdot)
  \end{equation*}
  pointwise for a.e. $t > 0$.
  Finally we deduce the existence of $\rho^*(t, \cdot)$ such that
  \begin{equation*}
    z_{\eps}(t, \cdot) \to z^*(t, \cdot) :=
    \left(\rho^*(t, \cdot), w^*(t, \cdot)
    \right)
  \end{equation*}
  pointwise for a.e. $t > 0$.

  Since $\modulo{\dot y_\eps(t)} \le [0, V_{\max}]$ for a.e. $t > 0$,
  Ascoli-Arzelà Theorem~\cite[Theorem~7.25]{MR0385023} implies that, 
  at least passing to a subsequence, there exist a continuous function
  $y^*$ such that
  \begin{equation*}
    y_\eps \to y^*
  \end{equation*}
  uniformly. Moreover by the estimates in \Cref{2}, we deduce that
  \begin{equation*}
    \tv\left(\dot y_\eps\right) \le \sup_{t>0}\mathcal F_v(t) + \tv(u_\eps)
    \le O(1) M.
  \end{equation*}
  Therefore we deduce that, at least passing to a subsequence,
 \begin{equation*}
    y_\eps \to y^*
  \end{equation*}
  in $\Wloc{1,1}\left([0, +\infty); \reali\right)$.

  It remains to prove that $\left(z^*, y^*\right)$
  satisfies \Cref{def:weak-solution-1}. The points 1., 2., and 3. are
  straightforward.

  \textbf{Point 4. of \Cref{def:weak-solution-1}}.
  We need to prove, for a.e. $t >0$,
  \begin{equation}
    \label{eq:p4}
    y^*(t) = y_{0} +
    \int_{0}^{t}\min \left\{u^*(s),
      v \left(s,y^*(s)+\right)\right\}\,ds\,.
  \end{equation}
  By construction we have, for a.e. $t >0$,
  \begin{equation}
    \label{eq:p4-eps}
    y_\eps(t) = y_{0} +
    \int_{0}^{t}\min \left\{u_\eps(s),
      v \left(s,y_\eps(s)+\right)\right\}\,ds\,.
  \end{equation}
  We have $u_\eps \to u^*$ and $y_\eps \to y^*$ a.e. as $\eps \to 0$.
  Moreover the curve $t \mapsto y^*(t)$ is non characteristic for the
  quantity $v$, since the Riemann coordinate $v$ travels with velocity
  given by $\lambda_1 < 0$; see~\ref{Hyp:H3}.
  Hence $v\left(s, y_\eps(s)\right) \to
  v\left(s, y^*(s)\right)$ for a.e. $s > 0$. Thus, passing to the limit
  in~\eqref{eq:p4-eps} as $\eps \to 0$, we deduce~\eqref{eq:p4}.

  \textbf{Point 5. of \Cref{def:weak-solution-1}}.
  We follow the same idea as in~\cite{GaravelloGoatinLiardPiccoli}, based on the fact that $z_\eps$ and $z^*$ are weak solutions
  to the PDE in~\eqref{eq:CP}.
  Without loss of generality, we assume that $z_\eps$ is a sequence with
  all the limit properties highlighted in the previous part of the proof.
  Fix $T > 0$, $\varphi \in \Cc1\left(]0, T[ \times \reali; \reali^+\right)$,
  and consider the sets
  \begin{equation*}
    \begin{split}
      D_l
      & = \left\{(t,x) \in [0, T] \times \reali:\, x < y^*(t) \right\}
      \\
      D_l^\eps
      & = \left\{(t,x) \in [0, T] \times \reali:\, x < y_\eps(t) \right\}
      \\
      I
      & = \left\{t > 0:\, \dot y^*(t) \textrm{ exists, }
        y_{\eps}(t) \to y^*(t) \textrm{ and } \dot y_{\eps}(t) \to \dot y^*(t)
        \textrm{ as } \eps \to 0\right\}.
    \end{split}
  \end{equation*}
  Note that the Lebesgue measure of $I$ is $T$.
  By~\cite[Theorem 2.2]{MR1702637}, for every $\eps$, we deduce that      
  \begin{equation*}
    \begin{split}
      & \quad
      \int_{D_l^\eps} 
      \left(\rho_\eps \partial_t \varphi + (\rho_\eps v(\rho_\eps, w_\eps))
        \partial_x \varphi\right)
      dt\,dx
      \\
      & = \int_0^T \rho_\eps\left(t, y_\eps(t)-\right)
      \left[v_\eps(t,y_\eps(t)-)) -
        \dot y_\eps(t)\right]\varphi(t,y_\eps(t)) dt
    \end{split}
  \end{equation*}
  and
  \begin{equation*}
    \begin{split}
      & \quad
      \int_{D_l} 
      \left(\rho^* \partial_t \varphi + (\rho^* v(\rho^*, w^*))
        \partial_x \varphi\right)
      dt\,dx
      \\
      & = \int_0^T \rho^* \left(t, y^*(t)-\right)
      \left[v^*(t,y^*(t)-)) -
        \dot y^*(t)\right]\varphi(t, y^*(t)) dt.
    \end{split}
  \end{equation*}
  The Dominated Convergence Theorem implies
  \begin{equation*}
    \begin{split}
      & \lim_{\eps\to 0}
      \int_{D_l^\eps} 
      \left(\rho_\eps \partial_t \varphi + (\rho_\eps v(\rho_\eps, w_\eps))
        \partial_x \varphi\right)
      dt\,dx
      \\
      =
      & \int_{D_l} 
      \left(\rho^* \partial_t \varphi + (\rho^* v(\rho^*, w^*))
        \partial_x \varphi\right)
      dt\,dx,
    \end{split}
  \end{equation*}
  so that
  \begin{equation*}
    \begin{split}
      & \int_0^T \rho^* \left(t, y^*(t)-\right)
      \left[v^*(t,y^*(t)-)) -
        \dot y^*(t)\right]\varphi(t, y^*(t)) dt
      \\
      =
      & \lim_{\eps \to 0}
      \int_0^T \rho_\eps\left(t, y_\eps(t)-\right)
      \left[v_\eps(t,y_\eps(t)-)) -
        \dot y_\eps(t)\right]\varphi(t,y_\eps(t)) dt
      \\
      =
      & \lim_{\eps \to 0}
      \int_I \rho_\eps\left(t, y_\eps(t)-\right)
      \left[v_\eps(t,y_\eps(t)-)) -
        \dot y_\eps(t)\right]\varphi(t,y_\eps(t)) dt.
    \end{split}
  \end{equation*}
  Define
  \begin{equation*}
    I_1 = \left\{t \in I: w_\eps(t, y_\eps(t)-) \to w^*(t, y^*(t))
      \textrm{ as } \eps \to 0\right\}.
  \end{equation*}
  The construction of approximate solutions implies that
  \begin{equation*}
    \begin{split}
      & \int_{I_1} \rho_\eps\left(t, y_\eps(t)-\right)
      \left[v_\eps(t,y_\eps(t)-)) -
        \dot y_\eps(t)\right]\varphi(t,y_\eps(t)) dt
      \\
      \le
      & \int_{I_1} F_{\alpha}(w_\eps(t, y_\eps(t)-), \dot y_\eps (t))
      \varphi(t,y_\eps(t)) dt
    \end{split}
  \end{equation*}
  and so, using the Dominated Convergence Theorem,
  \begin{equation*}
    \begin{split}
      & \liminf_{\eps \to 0} \int_{I_1} \rho_\eps\left(t, y_\eps(t)-\right)
      \left[v_\eps(t,y_\eps(t)-)) -
        \dot y_\eps(t)\right]\varphi(t,y_\eps(t)) dt
      \\
      \le
      & \int_{I_1} F_{\alpha}(w^*(t, y^*(t)-), \dot y^* (t))
      \varphi(t,y^*(t)) dt.
    \end{split}
  \end{equation*}
  Consider now $\bar t \in I \setminus I_{1}$.
  At least passing to a subsequence, we may assume that
  $w_\eps(\bar t, y_\eps(\bar t))$ is uniformly far from
  $w^*(\bar t, y^*(\bar t))$, which implies that the boundary at $\bar t$ is
  characteristic, so that
  \begin{equation*}
    v_\eps\left(\bar t, y_\eps(\bar t)\right) \to \dot y^*(\bar t)
  \end{equation*}
  as $\eps \to 0$.
  The Dominated Convergence Theorem implies that
  \begin{equation*}
    \begin{split}
      & \lim_{\eps \to 0} \int_{I \setminus I_1} \rho_\eps\left(t, y_\eps(t)-\right)
      \left[v_\eps(t,y_\eps(t)-)) -
        \dot y_\eps(t)\right]\varphi(t,y_\eps(t)) dt
      \\
      =
      & 0 \le \int_{I \infty I_1} F_{\alpha}(w^*(t, y^*(t)-), \dot y^* (t))
      \varphi(t,y^*(t)) dt.
    \end{split}
  \end{equation*}
  Therefore
  \begin{equation*}
    \begin{split}
      & \int_0^T \rho^* \left(t, y^*(t)-\right)
      \left[v^*(t,y^*(t)-)) -
        \dot y^*(t)\right]\varphi(t, y^*(t)) dt
      \\
      =
      & \lim_{\eps \to 0} \int_{I} \rho_\eps\left(t, y_\eps(t)-\right)
      \left[v_\eps(t,y_\eps(t)-)) -
        \dot y_\eps(t)\right]\varphi(t,y_\eps(t)) dt
      \\
      =
      & \liminf_{\eps \to 0} \int_{I_1} \rho_\eps\left(t, y_\eps(t)-\right)
      \left[v_\eps(t,y_\eps(t)-)) -
        \dot y_\eps(t)\right]\varphi(t,y_\eps(t)) dt
      \\
      \le
      & \int_{I_1} F_{\alpha}(w^*(t, y^*(t)-), \dot y^* (t))
      \varphi(t,y^*(t)) dt
      \\
      \le
      & \int_{I} F_{\alpha}(w^*(t, y^*(t)-), \dot y^* (t))
      \varphi(t,y^*(t)) dt.
    \end{split}
  \end{equation*}
  Since $\varphi \ge 0$, we deduce that
  \begin{equation*}
    \begin{split}
      & \rho^* \left(t, y^*(t)-\right)
      \left[v^*(t,y^*(t)-)) -
        \dot y^*(t)\right]\varphi(t, y^*(t))
      \\
      \le
      & F_{\alpha}(w^*(t, y^*(t)-), \dot y^* (t))
      \varphi(t,y^*(t))
    \end{split}
  \end{equation*}
  for a.e. $t \in [0, T]$, as required.
  A similar estimate holds for the right traces and so the point~5.
  of \Cref{def:weak-solution-1}
  is satisfied. The proof is concluded.
\end{proofof}

\section{Appendix: Technical Lemmas}
We state the following results about $\tilde v$.
\begin{lemma}
  \label{le:lip-v-rho}
  Assume~\ref{Hyp:H2} and fix $w \in [w_{\min}, w_{\max}]$. Then
  $\tilde v$ is Lipschitz continuous with respect to $\rho$ with
  Lipschitz constant $w C_\psi$.
\end{lemma}
\begin{proof}
  We have
  \begin{equation*}
    \modulo{\partial_\rho \tilde v} = \modulo{w \psi'(\rho)}
    \le w C_{\psi},   
  \end{equation*}
  completing the proof.
\end{proof}

\begin{lemma}
  \label{le:lip-rho-w}
  Assume~\ref{Hyp:H2} and fix $\tilde v > 0$. Then the function
  $w(\rho) = \frac{\tilde v}{\psi(\rho)}$ is invertible and the inverse function is
  Lipschitz continuous with
  Lipschitz constant $\frac{1}{\tilde v c_\psi}$.
\end{lemma}

\begin{proof}
  We have
  \begin{equation*}
    \modulo{w'\left(\rho\right)} = \frac{\tilde v}{\psi^2(\rho)}\modulo{\psi'(\rho)}
    \ge \frac{\tilde v}{\psi^2(0)} c_\psi = \tilde v c_\psi, 
  \end{equation*}
  concluding the proof.
\end{proof}

\begin{lemma}
  \label{tec_lem}
  Assume~\ref{Hyp:H2} and~\ref{Hyp:H3}.
  Fix $\bar w \in [w_{\min}, w_{\max}]$
  and consider, in the plane
  $\left(\rho, \rho v\right)$,
  $\rho \mapsto \mathcal L_1\left(\rho; \bar w\right) = \rho \bar w \psi(\rho)$
  the Lax curve of
  the first family with $w = \bar w$.
  Let $\rho_1 < \rho_2$ such that
  $\left(\rho_1, \mathcal L_1\left(\rho_1; \bar w\right)\right) \in C$
  and $\left(\rho_2, \mathcal L_1\left(\rho_2; \bar w\right)\right) \in C$.
  \begin{enumerate}
  \item Fix $u \in [0, V_{\max}[$. Define $F_1 > F_2$ such that
    \begin{equation}
      \label{eq:rho-L_1}
      \mathcal L_1\left(\rho_1; \bar w\right) = F_1 + u \rho_1
      \qquad \textrm{ and } \qquad
      \mathcal L_1\left(\rho_2; \bar w\right) = F_2 + u \rho_2.
    \end{equation}
    Then
    \begin{equation*}
      0 < \rho_2 - \rho_1 \le \frac{1}{\bar w \bar \lambda + u}
      \left(F_1 - F_2\right),
    \end{equation*}
    where $\bar \lambda$ is defined in~\ref{Hyp:H3}.
    
  \item Define {$u_1 > u_2$} such that
    \begin{equation}
      \label{eq:rho-L_1-2}
      \begin{split}
      \mathcal L_1\left(\rho_1; \bar w\right) = F_{\alpha}(\bar w, u_1)
      + u_1 \rho_1
      \\
      \mathcal L_1\left(\rho_2; \bar w\right) = F_{\alpha}(\bar w, u_2)
      + u_2 \rho_2.
    \end{split}
    \end{equation}
    Then
    \begin{equation*}
      \modulo{\rho_2 - \rho_1} \le \frac{R + F_{\alpha, 1}(\bar w)}
      {\bar w \bar \lambda + u_2}
      \left(u_1 - u_2\right),
    \end{equation*}
    where $\bar \lambda$ is defined in~\ref{Hyp:H3}.
  \end{enumerate}
\end{lemma}
\begin{proof}
  For case 1,
  consider the function
  \begin{equation*}
    G\left(\rho; \bar w, u\right) = \mathcal L_1\left(\rho; \bar w\right)
    - u \rho = \rho \bar w \psi(\rho) - u \rho.
  \end{equation*}
  Equation~\eqref{eq:rho-L_1} implies that
  \begin{equation*}
    G\left(\rho_1; \bar w, u\right) = F_1
    \qquad \textrm{ and } \qquad
    G\left(\rho_2; \bar w, u\right) = F_2,
  \end{equation*}
  so that we need to prove that $G$ is invertible with respect to $\rho$
  and that the inverse is Lipschitz continuous with Lipschitz constant
  $\frac{1}{\bar w \bar \lambda + u}$.
  We have
  \begin{equation*}
    \partial_\rho G\left(\rho; \bar w, u\right) = \bar w \frac{d}{d \rho}
    \left(\rho \psi(\rho)\right) - u \le - \bar w \bar \lambda - u < 0,
  \end{equation*}
  by~\ref{Hyp:H2} and~\ref{Hyp:H3}. Hence $G$ is invertible with respect to
  $\rho$.
  Moreover
  \begin{equation*}
    \partial_\rho G\left(\rho; \bar w, u\right) = \bar w \frac{d}{d \rho}
    \left(\rho \psi(\rho)\right) - u \ge \bar w \left(\psi(R) + R \psi'(R)
    \right) - u.
  \end{equation*}
  Therefore the inverse function satisfies
  \begin{equation*}
    - \frac1{\bar w \bar \lambda + u} \le
    \partial_F G^{-1}\left(F; \bar w, u\right)
    \le \frac{1}{\bar w \left(\psi(R) + R \psi'(R)
      \right) - u} < 0
  \end{equation*}
  and so
  \begin{equation*}
    \modulo{\partial_F G^{-1}\left(F; \bar w, u\right)} \le
    \frac1{\bar w \bar \lambda + u}
  \end{equation*}
  concluding the proof.

  In case 2, consider the function
  consider the function
  \begin{equation*}
    G\left(\rho; \bar w, u\right) = \mathcal L_1\left(\rho; \bar w\right)
    - u \rho  - F_\alpha(\bar w, u).
  \end{equation*}
  Equation~\eqref{eq:rho-L_1-2} implies that
  \begin{equation*}
    G\left(\rho_1; \bar w, u_1\right) = 0
    \qquad \textrm{ and } \qquad
    G\left(\rho_2; \bar w, u_2\right) = 0.
  \end{equation*}
  Using the Implicit Function Theorem and analogous estimates contained in
  the previous case, we easily conclude.  
\end{proof}

\subsection*{Acknowledgments}
\revv{The authors thank the anonymous referees for their precious suggestions.}

Both the authors were partially supported by the INdAM-GNAMPA 2022 project \emph{Evolution Equations: Wellposedness, Control and Applications}.

M. Garavello was partially supported by the Project  funded  under  the  National  Recovery  and  Resilience  Plan  (NRRP),  Mission  4 Component  2  Investment  1.4 -Call  for  tender  No.  3138  of  16/12/2021
of  Italian  Ministry  of  University and Research
funded by the European Union -NextGenerationEU.
Award  Number:  
  CN000023,  Concession  Decree  No.  1033  of  17/06/2022  adopted  by  the  Italian Ministry of University and Research, CUP: H43C22000510001, Centro Nazionale per la Mobilità Sostenibile.

{\small{

    \bibliographystyle{abbrv}

    \bibliography{model} }}
\end{document}